\documentclass[11pt]{article}
\usepackage[T1]{fontenc}
\usepackage{lmodern}
\usepackage[margin=1in]{geometry}
\usepackage{amsmath,amssymb,amsthm,mathtools,microtype,booktabs,array}
\usepackage[colorlinks=true,linkcolor=blue,citecolor=blue,urlcolor=blue]{hyperref}
\hypersetup{pdftitle={Resolvent characteristics and quadratic mixing of Kac walk},pdfauthor={Yunjiang Jiang},pdfsubject={Wasserstein and total-variation mixing of the coordinate-plane Kac walk}}
\allowdisplaybreaks[2]
\newtheorem{theorem}{Theorem}[section]
\newtheorem{lemma}[theorem]{Lemma}
\newtheorem{proposition}[theorem]{Proposition}
\newtheorem{corollary}[theorem]{Corollary}
\theoremstyle{definition}\newtheorem{definition}[theorem]{Definition}
\newtheorem{example}[theorem]{Example}
\theoremstyle{remark}\newtheorem{remark}[theorem]{Remark}
\DeclareMathOperator{\Tr}{Tr}

\DeclareMathOperator{\Ad}{Ad}
\DeclareMathOperator{\ad}{ad}
\DeclareMathOperator{\rank}{rank}
\DeclareMathOperator{\diver}{div}

\newcommand{\E}{\mathbb E}
\newcommand{\PP}{\mathbb P}
\newcommand{\SO}{\mathrm{SO}}
\newcommand{\g}{\mathfrak{so}(n)}
\newcommand{\R}{\mathcal R}
\newcommand{\op}{\mathrm{op}}
\newcommand{\HS}{\mathrm{HS}}
\newcommand{\TV}{\mathrm{TV}}
\newcommand{\ip}[2]{\langle#1,#2\rangle}
\newcommand{\norm}[1]{\left\lVert#1\right\rVert}
\newcommand{\one}{\mathbf1}
\title{Resolvent characteristics and quadratic mixing\\of Kac's walk on $\mathrm{SO}(n)$}
\author{Yunjiang Jiang}
\date{}
\begin{document}
\maketitle
\begin{abstract}
We study the coordinate-plane Kac walk on $\mathrm{SO}(n)$ with independent uniform rotation angles. For the unnormalized Hilbert--Schmidt Riemannian distance, we prove that the Wasserstein-$2$ Lipschitz coefficient of a block of $2\binom n2$ steps is at most $2n^{-1/200}$ for sufficiently large $n$. This gives a fixed-accuracy $O(n^2)$ Wasserstein mixing bound and inverse-polynomial accuracy in a constant number of blocks. We also obtain an $O(n^2)$ total-variation upper bound. The coupling averages each component of a regularized-inverse angle correction over its own angle, so its exact flow preserves the joint angle law. A deterministic decreasing regularization parameter cancels the main scalar and matrix-valued resolvent drifts. Its stability factor is polynomial in the inverse regularization, and the matrix estimates close on two partial traces. The same scalar estimate bounds the fraction of deficient covariance directions; independent cores and a cutoff integration-by-parts argument give the total-variation transfer. We provide the parameter estimates, conditioning arguments, and finite-time singular-mass treatment explicitly. The results imply total-variation pre-cutoff on the quadratic scale, but do not establish either the presence or the absence of cutoff.
\end{abstract}
\begingroup
\small
\tableofcontents
\endgroup
\clearpage

\section{Introduction and results}\label{sec:introduction}
Put $N=\binom n2$, $n\ge5$. The transition kernel $P_n$ left-multiplies by a uniformly selected coordinate-plane rotation $R_{ij}(\theta)$, where $\theta$ is uniform modulo $2\pi$. All updates are independent. Haar probability on $\SO(n)$ is denoted by $\pi_n$. Let $d$ be the bi-invariant Riemannian distance induced by the unnormalized Hilbert--Schmidt inner product on $\g$. Set $W_2=W_{2,d}$ and use $\norm{\mu-\nu}_{\TV}=\sup_E|\mu(E)-\nu(E)|$ for probabilities.

\begin{theorem}\label{thm:main}
For all sufficiently large $n$, the Wasserstein Lipschitz coefficient of a $2N$-step block satisfies
\begin{equation}\label{eq:main-kappa}
 \kappa_n:=\sup_{x\ne y}
 \frac{W_2(P_n^{2N}(x,\cdot),P_n^{2N}(y,\cdot))}{d(x,y)}
 \le 2n^{-1/200}\le n^{-1/300}.
\end{equation}
Consequently, for every fixed $A>0$, a constant number of such blocks gives
\[
 \sup_x W_2(P_n^{\lceil C_A N\rceil}(x,\cdot),\pi_n)\le n^{-A}.
\]
There is a sequence $T_n=O(N)$ for which
\[
 \sup_x\norm{P_n^{T_n}(x,\cdot)-\pi_n}_{\TV}\le Cn^{-5}.
\]
In particular, for fixed accuracy $\epsilon>0$ in Wasserstein distance, and $\epsilon\in(0,1)$ in total variation,
\[
 t_{\mathrm{mix},W_2}(\epsilon)=O_\epsilon(n^2),\qquad
 t_{\mathrm{mix},\TV}(\epsilon)=O_\epsilon(n^2).
\]
The Wasserstein upper bound also holds for Hilbert--Schmidt chord distance and for $W_1$.
\end{theorem}

For precision, the two mixing-time conventions in the theorem are
\[
\begin{aligned}
 t_{W_2}^{(n)}(\epsilon)
 &=\min\{t\ge0:\sup_xW_2(P_n^t(x,\cdot),\pi_n)\le\epsilon\},\\
 t_{\TV}^{(n)}(\epsilon)
 &=\min\{t\ge0:\sup_x\|P_n^t(x,\cdot)-\pi_n\|_{\TV}\le\epsilon\}.
\end{aligned}
\]
Time is discrete: one step changes exactly one uniformly selected coordinate plane. Every block used in an actual coupling contains exactly $2N$ such steps. Auxiliary rotations introduced in expectation estimates are not added to the walk.

\begin{corollary}[Quadratic scale does not decide cutoff]\label{cor:precutoff}
For every fixed $\epsilon\in(0,1)$ and all sufficiently large $n$,
\[
 N\le t_{\TV}^{(n)}(\epsilon)\le26600N.
\]
In particular, the family has total-variation pre-cutoff. The estimates do not imply that the ratio
$t_{\TV}^{(n)}(\epsilon)/t_{\TV}^{(n)}(1-\epsilon)$ tends to one, or that it fails to do so. Thus no conclusion of ``no cutoff'' follows. Definitions, the proof of this corollary, and two contrasting examples are given in Section~\ref{sec:cutoff}.
\end{corollary}

The exponents in this note are deliberately conservative. The required conclusion is a block coefficient bounded by any fixed negative power of $n$. Neither the constants nor the sufficiently-large-$n$ threshold are optimized. A dimension argument gives $t_{\mathrm{mix},\TV}(\epsilon)\ge N$ for $\epsilon<1$; see Section~\ref{sec:tv}.

The classical inputs are the group spectral gap of Carlen--Carvalho--Loss~\cite{CCL} and the Golden--Thompson inequality~\cite{Golden,Thompson}. We state the exact normalization of each when used. Elementary Hilbert-space martingale estimates and the required local-to-global transportation argument are proved in the forms needed below; the latter originates in Oliveira~\cite{Oliveira}. The proof does not assume a previous improvement of the mixing time. It also does not assume a limiting random-matrix law, a minimum covariance eigenvalue for one core, or absolute continuity of the full finite-time law.

Section~\ref{sec:history} gives a historical account, distinguishing the physical sphere model, the coordinate-plane matrix walk, and related rotation processes. For context, Oliveira~\cite{Oliveira} established an $O(n^2\log n)$ Wasserstein bound and a quadratic lower bound. Pillai and Smith~\cite{PS18} proved a polynomial total-variation upper bound through a non-Markovian coupling. Their later preprint~\cite{PS26} obtains $O(n^2\log n)$ in total variation. Our proof uses an angle-space coupling as well, but enforces the noise law by an exactly divergence-free flow rather than an approximately law-preserving additive perturbation.

\subsection{Why a resolvent characteristic is useful}
An endpoint is a smooth function of its $m$ angles. If $J$ is its left-trivialized Jacobian and $JJ^*=2B$, a preliminary correction in the direction $h$ is
\[
 v=-\tfrac12J^*(B+\eta I)^{-1}h.
\]
It leaves the formal endpoint displacement $\eta(B+\eta I)^{-1}h$. This would be small if the average regularized inverse were controlled. However, $v$ generally changes the joint uniform angle distribution. The first issue is therefore to construct a correction with exactly the right marginal law, not merely to make its formal displacement small.

We replace the $s$th component by its own-angle average. The resulting vector field is divergence free, and hence its flow is exactly measure preserving. Circle Poincare bounds the resulting error by the sum of the individual squared angle derivatives of the resolvent. This sum is the central probabilistic quantity of the proof.

Differentiating the regularized inverse produces both $(S+zI)^{-1}$ and a colored matrix $Y(S+zI)^{-1}$. Their rank-one update equations would usually create higher inverse powers. We avoid a hierarchy by allowing $z$ to decrease on a deterministic path. This cancels the leading drift. The only remaining higher inverse powers are multiplied by a scalar trace error, which is already small. The scalar stability factor telescopes to $z_0/\eta$. Retaining this factor, rather than applying a crude exponential Gronwall estimate, is essential.

\subsection{A reading map and the quantitative ledger}
The historical section can be read independently of the proof. The notation and classical inputs are in Section~\ref{sec:elementary}. Section~\ref{sec:comparison} proves precisely which resolvent tests can be averaged over an auxiliary rotation. The scalar characteristic estimate is proved in Section~\ref{sec:characteristic}; the two matrix-valued estimates are in Section~\ref{sec:matrix}. Sections~\ref{sec:two-block} and~\ref{sec:influence} turn them into an individual-angle bound. The actual coupling, with no auxiliary time cost, is constructed in Section~\ref{sec:coupling}. Section~\ref{sec:tv} proves the separate total-variation transfer.

Here is the ledger of bounds, so that the powers can be checked before reading the proofs. Throughout,
\[
 \ell=\log(en),\qquad R=\lceil n^{5/4}\rceil,\qquad n^{-1/100}\le\eta\le1.
\]
The constants are uniform in this range and in all deterministic coloring schedules with at most $4N$ actual steps.
\begin{center}
\small
\begin{tabular}{p{0.60\textwidth}p{0.31\textwidth}}
\toprule
Quantity & Upper bound, up to an absolute constant\\
\midrule
Resolvent coordinate-diagonal variance & $n^{-3/4}\eta^{-2}$\\
Maximal scalar characteristic error in $L^2$ & $n^{-3/8}\eta^{-4}$\\
$L^2$ norm of the centered resolvent partial trace & $n^{5/4}\eta^{-6}$\\
$L^2$ norm of the centered colored-resolvent partial trace & $n^{5/4}\ell\eta^{-6}$\\
Independent-block squared commutator average & $n^{-1/2}\ell^4\eta^{-12}$\\
Sum of individual squared angle influences, divided by $N$ & $n^{-1/2}\ell^8\eta^{-16}$\\
Mean squared endpoint cost of restoring the angle law & $n^{-1/2}\ell^{12}\eta^{-20}\|h\|^2$\\
Mean squared formal endpoint residual & $2\eta\|h\|^2$\\
\bottomrule
\end{tabular}
\end{center}
At $\eta=n^{-1/100}$, every error still carries a fixed negative power of $n$. We deliberately leave large reserves in the exponents; numerical optimization is not used anywhere.

The final quantitative statement before choosing $\eta$ is
\begin{equation}\label{eq:intro-uniform-block}
 \kappa_n\le\sqrt{2\eta}+Cn^{-1/4}\ell^6\eta^{-10}.
\end{equation}
A fixed negative power of $n$ for $\kappa_n$ is all that is needed: the initial distance is only of order $\sqrt n$, so a constant number of blocks reaches any fixed inverse-polynomial accuracy.

\section{Historical background and related work}\label{sec:history}
The name ``Kac's walk'' is used for several related processes. Their connections explain both the physical motivation and the mathematical interest of the problem, but their state spaces, time normalizations, and notions of convergence must be kept separate. This section records the direct spectral and mixing-time line of work, together with the kinetic-theory and algorithmic developments most relevant to it. Preprint and withdrawal statuses below refer to the public record checked on 7 September 2026. A cited preprint is not thereby a published theorem, and a discussion question is not used as a mathematical input.

\subsection{The kinetic model and the matrix lift}
Kac introduced a stochastic collision model in \emph{Foundations of Kinetic Theory}~\cite{Kac56}, and developed its probabilistic setting further in~\cite{Kac59}. In its simplest version, the coordinates of a vector represent one-dimensional particle velocities. A collision replaces a selected pair by
\[
 (v_i,v_j)\longmapsto
 (v_i\cos\theta-v_j\sin\theta,\ v_i\sin\theta+v_j\cos\theta).
\]
This preserves $v_i^2+v_j^2$, and hence the total kinetic energy. The finite-particle state space is therefore an energy sphere. The model deliberately simplifies physical binary collisions: in this one-dimensional form it conserves energy but not total momentum. Its purpose is to isolate the probabilistic mechanism of binary interaction and the emergence of a nonlinear kinetic equation, not to reproduce every conservation law of a gas.

A central question in Kac's program is \emph{propagation of chaos}. Roughly, when a symmetric $n$-particle law initially has asymptotically independent fixed-size marginals, one asks whether this property persists and whether the one-particle marginal follows the appropriate nonlinear collision equation. This limit, with the number of observed particles held fixed while $n\to\infty$, is different from approximation of the entire $n$-particle law in total variation. McKean studied approach to equilibrium for the Kac model and probabilistic representations of the Maxwellian Boltzmann equation~\cite{McKean66,McKean67}. Gr\"unbaum developed propagation-of-chaos results for the Boltzmann equation~\cite{Grunbaum71}. The later account of Carlen--Carvalho--Loss~\cite{CCL11} explains the relation among the finite-particle master equation, the nonlinear kinetic equation, and quantitative relaxation.

Entropy introduces another important distinction. Carlen, Carvalho, Le Roux, Loss, and Villani~\cite{CCLRLV10} studied entropy and chaoticity in the Kac model. Einav~\cite{Einav11} and Carlen--Carvalho--Einav~\cite{CCE18} investigated entropy production and its dimension dependence. Mischler--Mouhot~\cite{MM13} developed a quantitative treatment of Kac's program for collision models, while Hauray--Mischler~\cite{HM14} studied several precise formulations of chaos and the relations between them. These works concern substantive questions about the passage from many-particle dynamics to kinetic equations. A worst-start matrix mixing bound does not replace their conclusions, and an entropy or chaos estimate cannot be converted into the matrix mixing statement here without an additional argument.

The present matrix walk is a lift of the sphere process. If $F_t$ is its random product, then $F_t x$ evolves as the sphere walk for every fixed unit vector $x$. Recording only $F_t e_1$ observes one column; recording $F_t$ observes a complete orthonormal frame and all the dependence between its columns. The sphere has dimension $n-1$, whereas $\SO(n)$ has dimension $N=\binom n2$. This geometric difference is already reflected in the distinct singularity obstructions from point starts. The use of plane rotations to sample random orthogonal matrices also appears in Hastings's discussion of Monte Carlo sampling~\cite{Hastings70}.

We count one pair collision as one discrete step. Kinetic-theory generators are often accelerated so that every particle experiences order-one collisions per unit time. Statements about a dimension-independent gap in such a convention are consistent with a gap of order $1/n$ for our discrete kernel. Metric normalization matters as well: Wasserstein accuracy in the unnormalized matrix distance used here is not the same requirement as accuracy in that distance divided by $\sqrt n$.

\subsection{From spectral gaps to mixing from singular starts}
Diaconis and Saloff-Coste~\cite{DSC00} obtained quantitative estimates for Kac's master equation and the associated rotation walks, using comparison methods. Janvresse~\cite{Janvresse01} established the correct spectral-gap order for the sphere model; her later work~\cite{Janvresse03} treated semigroups of random rotations on $\SO(n)$. Maslen~\cite{Maslen03} analyzed the spectrum by representation-theoretic methods. Carlen--Carvalho--Loss~\cite{CCL} determined the exact gap and, in their treatment of the group process, identified it with the sphere gap. Caputo~\cite{Caputo08} subsequently gave a reduction method for spectral-gap calculations in binary-collision systems, including the Kac walk.

For the kernel in this paper, the exact formula is
\[
 \gamma_n=\frac{n+2}{2n(n-1)}\asymp n^{-1}.
\]
This is an input to our proof. It does not alone give a worst-start mixing estimate: a point mass has no square-integrable density relative to Haar. Moreover, the full endpoint law can retain a singular component at every finite time, for example on the event that every chosen coordinate plane is the same. Spectral contraction is powerful once a suitable regularity or comparison estimate is available, but establishing that estimate is part of the problem.

For the full coordinate-plane matrix walk, the weak-distance bounds of Diaconis--Saloff-Coste were followed by the unpublished 2007 preprint of Pak--Sidenko~\cite{PakSidenko07}, which improved the weak mixing upper bound from order $n^4\log n$ to order $n^{5/2}\log n$. Sidenko's 2008 dissertation~\cite{Sidenko08} also records this work. Oliveira's preprint of 2007, published in 2009~\cite{Oliveira}, proved the $O(n^2\log n)$ Wasserstein bound, together with a quadratic lower bound. Its local-to-global transportation theorem provides a way to turn infinitesimal couplings into global contraction; Section~\ref{sec:coupling} proves the specialization used here.

Total variation required additional ideas. Jiang~\cite{Jiang17} proved polynomial mixing for the coordinate-plane walk on the special orthogonal group, using spectral information together with quantitative control of the endpoint parametrization. Pillai--Smith~\cite{PS18}, first circulated in 2016 and published in 2018, obtained an $O(n^4\log n)$ upper bound through a non-Markovian coupling and estimates for a dependent random-matrix linearization. Their April 2026 preprint~\cite{PS26} gives an $O(n^2\log n)$ total-variation upper bound. It first contracts two copies in Wasserstein distance and then analyzes a Gaussian perturbation of the angle noise through a nondegenerate linearization. The preprint conjectures that the logarithmic factor in its upper bound is necessary; it does not prove a matching logarithmic lower bound. The connection to discrete Malliavin-type estimates is particularly close to the covariance viewpoint used here, although our marginal-preservation step and resolvent estimates are different.

The following summary concerns the \emph{full coordinate-plane matrix walk}, not the sphere or a conjugation-invariant variant. The historical weak-distance entries use the conventions of the cited papers; they should not be read as sharp constants under a common accuracy normalization.
\begin{center}
\small
\begin{tabular}{>{\raggedright\arraybackslash}p{0.25\textwidth}>{\raggedright\arraybackslash}p{0.24\textwidth}>{\raggedright\arraybackslash}p{0.36\textwidth}}
\toprule
Work & Type of control & Upper bound or role\\
\midrule
Diaconis--Saloff-Coste (2000)~\cite{DSC00} & Weak-distance estimates & $O(n^4\log n)$; comparison framework\\[3pt]
Pak--Sidenko (2007), preprint~\cite{PakSidenko07} & Weak-distance estimates & $O(n^{5/2}\log n)$\\[3pt]
Oliveira (2009)~\cite{Oliveira} & Wasserstein & $O(n^2\log n)$; quadratic lower bound\\[3pt]
Jiang (2017)~\cite{Jiang17} & Total variation & Polynomial upper bound\\[3pt]
Pillai--Smith (2018)~\cite{PS18} & Total variation & $O(n^4\log n)$\\[3pt]
Pillai--Smith (2026), preprint~\cite{PS26} & Total variation & $O(n^2\log n)$\\
\bottomrule
\end{tabular}
\end{center}
The table excludes withdrawn claims and does not assign a published status to a preprint. Theorem~\ref{thm:main} addresses the remaining quadratic-scale upper-bound question for this same coordinate-plane matrix kernel. It does not identify a sharp leading constant or a limiting profile.

\subsection{The sphere, redistribution chains, and cutoff}
The sphere mixing problem has its own chronology. Jiang~\cite{Jiang12} proved an $O(n^5(\log n)^3)$ total-variation upper bound for the sphere walk started from a coordinate vector. Pillai--Smith~\cite{PS17}, circulated in 2015 and published in 2017, established order $n\log n$ mixing from a worst-case starting point, with explicit constant-factor upper and lower bounds. These are sphere results and should not be entered as full-matrix total-variation bounds.

The July 2026 preprint of Jain--Mizgerd~\cite{JM26} proves total-variation cutoff for the sphere walk \emph{from a coordinate-vector start}, at
\[
 C_{\mathrm{BRW}}n\log n,\qquad C_{\mathrm{BRW}}\approx3.8916.
\]
The constant is determined by an extremal speed in a branching random walk; it is not the previously conjectured constant $2$. The stated starting condition is important. Their theorem does not assert worst-start cutoff for the full matrix walk. In particular, taking one column is a projection of our chain, and mixing of that projection neither determines nor rules out a later transition in the distribution of all columns.

There is also a useful energy-chain interpretation. Squaring the coordinates of the sphere walk produces a mass-redistribution process: the energy in a selected pair is split in $\operatorname{Beta}(1/2,1/2)$ proportions. The stationary energy distribution is $\operatorname{Dirichlet}(1/2,\ldots,1/2)$. This links the Kac model to exchange dynamics. Caputo--Quattropani--Sau~\cite{CQS25} study Wasserstein cutoff for a broad class of mean-field exchange models, including noisy redistribution dynamics. Their metric and state variable are not the matrix metric and full frame considered here. Related averaging and redistribution models are informative comparisons, not interchangeable formulations of matrix total variation.

A second family of comparisons consists of \emph{conjugation-invariant} rotation walks. Rosenthal~\cite{Rosenthal94} analyzed random rotations using characters, and Porod~\cite{Porod96} established cutoff results for random reflections. Hough--Jiang~\cite{HJ17} proved cutoff for the uniform-plane Kac walk, in which the plane is chosen uniformly from the space of two-dimensional subspaces rather than from the fixed coordinate planes. The increment measure in that setting is invariant under conjugation, allowing a different harmonic-analysis approach. The coordinate-plane kernel is invariant under signed coordinate permutations but not under all conjugations; the character estimates for the uniform-plane walk therefore do not directly apply.

These distinctions also explain why a quadratic upper bound is not evidence for ``no cutoff.'' Both a sharp transition within the quadratic scale and a non-sharp transition are compatible with Theorem~\ref{thm:main}. Section~\ref{sec:cutoff} states the exact pre-cutoff consequence and gives complete elementary examples of both possibilities on a quadratic scale. In kinetic theory, the term \emph{angular cutoff} has a different meaning, concerning the collision kernel; it should not be confused with mixing-time cutoff.

\subsection{Sampling, restricted observables, and algorithmic uses}
A coordinate rotation is inexpensive to apply and is naturally compatible with orthogonality constraints. This makes repeated plane rotations an appealing sampling primitive, but different applications may require substantially less than full total-variation approximation to Haar. Hastings~\cite{Hastings70} provides an early Monte Carlo context. Mitchell~\cite{Mitchell08} studies deterministic sampling of rotation groups by successive orthogonal images, a related geometric sampling question rather than a mixing theorem for the present random walk.

Jain, Pillai, Sah, Sawhney, and Smith~\cite{JPSSS22} use Kac's walk for fast, memory-efficient dimension reduction. Their work belongs to the line of fast Johnson--Lindenstrauss constructions initiated by Ailon--Chazelle~\cite{AilonChazelle06}. For dimension reduction, the relevant question is preservation of the geometry of a prescribed collection of vectors. This requirement is not equivalent to closeness of the whole random matrix law in total variation.

The preprint of Pillai, Smith, and Vaikuntanathan~\cite{PSV26} makes this distinction explicit by studying fixed-column and low-degree polynomial observables. It proves that such restricted tests can become Haar-like on scales shorter than full-matrix mixing. Vaikuntanathan--Zamir~\cite{VZ26} develop algorithmic uses of trapdoored matrices and restricted forms of randomness. These results help explain why Kac-type products are useful even before a full mixing theorem is available. They are related motivation, not external estimates assumed in the proof below.

Random-rotation ideas also appear in quantum information. Helsen, Nezami, Reagor, and Walter~\cite{HNRW22} study benchmarking of matchgate circuits using their orthogonal-group representation. Lu, Qin, Song, Yao, and Zhao~\cite{LQSYZ25} analyze a parallel Kac construction for pseudorandom unitaries. The latter uses a parallel unitary process and computational indistinguishability, not the discrete coordinate-plane $\SO(n)$ chain and information-theoretic TV distance here. Its update count cannot be compared directly with ours. Such applications illustrate the range of uses of rotation-based randomization without making any claim that the present worst-start bound is necessary for every application.

\subsection{Unpublished work, public questions, and the status of quadratic claims}
The unpublished record is relevant to the development of the problem. In addition to the Pak--Sidenko preprint and Sidenko dissertation already cited, the 2026 papers~\cite{PS26,JM26,PSV26} are identified here as preprints. Their public statements concern, respectively, the full matrix TV bound, sphere cutoff from a coordinate start, and restricted-observable mixing. The different formulations should not be compressed into a single unqualified claim about ``Kac mixing.''

Two MathOverflow discussions from 2011 record concrete conceptual questions about the matrix problem. In \emph{Decomposition of Haar measure other than Hurwitz's}~\cite{MOHaar}, the question concerns alternative products of Givens rotations, the angular distributions in a Haar parametrization, and the behavior of Jacobians. The discussion points to geometric sampling constructions, including~\cite{Mitchell08}. In \emph{rotation along a coordinate axis and Kac random walk}~\cite{MOKac}, the question concerns the interpretation of auxiliary subgroup averages in Maslen's spectral analysis. Both were posted under the name John Jiang. These discussions are useful records of the distinction between deterministic parametrizations, random coordinate words, and representation averages; neither is invoked as a theorem in this paper.

Liu's preprint~\cite{LiuWithdrawn26}, first posted on 29 August 2026, claimed quadratic total-variation mixing using discrete Malliavin and derivative-shell estimates. It was withdrawn by its author on 1 September 2026; the stated reason is that the main result needs further rigorous verification. We record that claim and its status because it concerns the same proposed scale. We do not treat its conclusion as established, infer a specific mathematical error from the withdrawal, or use any of its covariance or derivative estimates. The proof below supplies its own scalar and matrix-valued resolvent bounds, marginal-preserving coupling, and smoothing argument.

\subsection{What the matrix problem isolates}
The physical and probabilistic significance of the problem is therefore specific. On the physical side, the same local, energy-preserving collision operation is a basic part of Kac's finite-particle program. On the probabilistic side, the full matrix walk is a concrete high-dimensional chain with singular local transitions, a known spectral gap, and a strong nonlinear constraint. It tests how local randomness develops enough regularity to randomize an entire frame, and how transportation estimates can be combined with a separately justified TV smoothing step. These questions are related to, but distinct from, propagation of chaos, entropy production, algorithmic pseudorandomness, and cutoff for projected or conjugation-invariant walks. The proof will keep those distinctions, the update normalization, and every additional regularity input explicit.

\section{Covariances, partial traces, and elementary estimates}\label{sec:elementary}
Use the orthonormal basis
\[
 a_{ij}=(E_{ij}-E_{ji})/\sqrt2,\qquad i<j,
\]
of $\g\simeq\Lambda^2\mathbb R^n$. Write $\rho(U)=\Ad_U$, $P_a=aa^*$, $L_a=\ad_{\sqrt2a}$, and $\tau(A)=\Tr(A)/N$ for operators on $\g$. For a coordinate generator,
\begin{equation}\label{eq:coordinate-rotation}
 L_a^3=-L_a,\quad \norm{L_a}_{\op}=1,\quad
 O_a(\theta):=\rho(R_a(\theta))
 =I+\sin\theta L_a+(1-\cos\theta)L_a^2,\quad O_a a=a.
\end{equation}
Also $\norm{(O_a-I)z}\le2\norm{L_a z}$: $L_a$ is an isometry on its image, is zero on its kernel, and $O_a-I$ has norm at most two on the image.

For a block $F_m=R_m\cdots R_1$, its right covariance is
\[
 D_m=\sum_{s=1}^m v_sv_s^*,\qquad
 v_s=\rho(R_m\cdots R_{s+1})a_{e_s}.
\]
It satisfies the exact recursion
\begin{equation}\label{eq:cov-recursion}
 D_{t+1}=O_a(D_t+P_a)O_a^*.
\end{equation}
The left covariance is
\[
 B_m=\sum_s w_sw_s^*,\qquad
 w_s=\rho((R_{s-1}\cdots R_1)^{-1})a_{e_s}.
\]
It is pointwise conjugate to $D_m$ and has the same marginal distribution, by reversing and inverting all the independent rotations. A column $w_s$ has zero derivative in its own angle.

We also allow a deterministic two-color schedule. Its total retained covariance $S_t$ and one colored covariance $Y_t$ obey
\begin{equation}\label{eq:colored}
 S_{t+1}=O_a(S_t+\varepsilon_tP_a)O_a^*,\qquad
 Y_{t+1}=O_a(Y_t+\chi_tP_a)O_a^*,
 \quad 0\le\chi_t\le\varepsilon_t\le1,
\end{equation}
where $\varepsilon_t,\chi_t$ are deterministic zero-one variables and $S_0=Y_0=0$. All processes considered in probabilistic lemmas have at most $4N$ actual rotations. Inactive steps have $\varepsilon_t=\chi_t=0$; they still rotate both covariances. Always $0\preceq Y_t\preceq S_t$.

The filtration $\mathcal F_t$ consists of the complete input pairs and angles through time $t$. The update at index $t$ uses a fresh generator $a=a_{e_{t+1}}$ and angle $\theta_{t+1}$ independent of $\mathcal F_t$. The quantities $\varepsilon_t,\chi_t$ and the regularization path used later are deterministic, so conditioning on $\mathcal F_t$ does not change their values. A displayed conditional expectation over $a$ refers to its uniform distribution on the coordinate orthonormal basis.

We distinguish three norms. The norm on $\mathfrak{so}(n)$ is the fundamental Frobenius norm; on operators on this $N$-dimensional space, $\|\cdot\|_{\op}$ and $\|\cdot\|_{\HS}$ are operator and Hilbert--Schmidt norms. The output of the partial trace is an $n\times n$ matrix and uses its Frobenius norm $\|\cdot\|_F$. We write $\|Z\|_{L^2(F)}=(\E\|Z\|_F^2)^{1/2}$. An unadorned norm of a covariance or resolvent means operator norm, except inside an explicitly labeled $L^2(F)$ expression. Nuclear norms of operators are denoted by $\|\cdot\|_{S_1}$.

\subsection{The two trivializations and their marginal laws}
The differential of the endpoint product is
\[
 \partial_{\theta_s}F_m
 =R_m\cdots R_{s+1}(\sqrt2a_{e_s})R_s\cdots R_1.
\]
Multiplication on the right by $F_m^{-1}$ gives $\sqrt2v_s$, while multiplication on the left by $F_m^{-1}$ gives $\sqrt2w_s$. The factor $R_s$ disappears in the latter conjugation because it fixes $a_{e_s}$. In particular, $w_s$ depends on earlier rotations and on $e_s$, but not on $\theta_s$ or later rotations.

To verify equality of the covariance laws without assuming independence of their entries, transform the entire input word by
\[
 (e_1,\theta_1),\ldots,(e_m,\theta_m)
 \longmapsto (e_m,-\theta_m),\ldots,(e_1,-\theta_1).
\]
This preserves the product input distribution. Its suffixes are the inverse prefixes of the original word. The transformed right rank-one columns therefore agree, in reverse order, with the original left rank-one columns. Any sign change of a generator cancels in its projection. This proves equality in distribution; pointwise conjugacy gives equality of their spectra for the original word.

\begin{lemma}\label{lem:norm-tail}
For these processes,
\begin{equation}\label{eq:exp-tail}
 \E\tau(e^{\theta S_t})\le e^{4(e^\theta-1)},\qquad
 \PP\{\max_{t\le4N}\norm{S_t}_{\op}>u\}\le Ne^{4(e-1)-u}.
\end{equation}
In particular every fixed moment of $1+\norm{S_t}_{\op}$ is $O((\log(en))^j)$ at order $j$.
\end{lemma}
\begin{proof}
Golden--Thompson bounds $\Tr e^{\theta(S+P_a)}$ by $\Tr(e^{\theta S}e^{\theta P_a})$. Since $\E_a P_a=I/N$,
\[
 \E_a e^{\theta P_a}=(1+(e^\theta-1)/N)I.
\]
Iteration proves the first assertion. The eigenvalues increase under a positive rank-one addition and do not change under conjugation. The final norm is therefore the path supremum. Exponential Markov inequality with $\theta=1$ proves the tail; integrate it to obtain the moment assertion.
\end{proof}

\subsection{Matrix facts used below}
We record the relevant matrix inequalities with their norms and normalizations. Golden--Thompson, the classical trace inequality used in Lemma~\ref{lem:norm-tail}, states that
\[
 \Tr e^{A+B}\le\Tr(e^Ae^B)
\]
for finite-dimensional self-adjoint $A,B$; see~\cite{Golden,Thompson}. The matrices on the right need not commute. In its conditional uses later, an upper bound $\E(e^B\mid\mathcal F)\preceq cI$ can be paired with the positive matrix $e^A$ inside the trace when $A$ is $\mathcal F$-measurable.

\begin{lemma}[Rank-one and low-rank inverse identities]\label{lem:inverse-facts}
Let $S,C$ be positive-semidefinite operators on a finite-dimensional real inner-product space, let $z>0$, and put $T=(S+zI)^{-1}$.
For any vector $a$,
\begin{equation}\label{eq:SM-elementary}
 (S+aa^*+zI)^{-1}
 =T-\frac{(Ta)(Ta)^*}{1+\langle a,Ta\rangle}.
\end{equation}
If $T_C=(S+C+zI)^{-1}$ and $\rank C\le q$, then
\begin{equation}\label{eq:low-rank-elementary}
 0\preceq T-T_C\preceq z^{-1}I,\qquad
 \rank(T-T_C)\le q,\qquad
 \|T-T_C\|_{\HS}^2\le qz^{-2}.
\end{equation}
Moreover, $\partial_z T=-T^2$; if $0<d<z$ and $z-d\ge\eta>0$, then
\begin{equation}\label{eq:shift-Taylor-elementary}
 (S+(z-d)I)^{-1}=T+dT^2+E,
 \qquad \|E\|_{\op}\le d^2\eta^{-3}.
\end{equation}
\end{lemma}
\begin{proof}
Multiplying the right side of~\eqref{eq:SM-elementary} by $S+aa^*+zI$ gives the identity, since $(S+zI)Ta=a$ and the scalar denominator is positive.
For a general $C$, write $K=C^{1/2}$. Direct multiplication gives
\[
 T-T_C=TK(I+K^*TK)^{-1}K^*T.
\]
This is positive semidefinite and of rank at most $\rank K=\rank C$. Since $T_C\succeq0$ and $T\preceq z^{-1}I$, the operator-norm bound follows. A rank-$q$ operator whose norm is at most $z^{-1}$ has squared Hilbert--Schmidt norm at most $qz^{-2}$.
Differentiating $(S+zI)T=I$ proves $\partial_zT=-T^2$ and then $\partial_z^2T=2T^3$. Taylor's integral remainder gives
\[
 E=2\int_0^d(d-u)(S+(z-u)I)^{-3}\,du.
\]
Its norm is at most $2\eta^{-3}\int_0^d(d-u)\,du=d^2\eta^{-3}$.
\end{proof}

\begin{lemma}[Accumulating centered increments before a rotation]\label{lem:rotated-energy}
Let $\mathcal H$ be a finite-dimensional Hilbert space. Suppose
\[
 Z_{t+1}=\mathcal O_t(Z_t+b_t+\xi_t),
\]
where $\mathcal O_t$ is an orthogonal operator, $Z_t,b_t$ are measurable with respect to the past $\mathcal F_t$, and $\E(\xi_t\mid\mathcal F_t)=0$. The fresh orthogonal operator may depend on $\xi_t$. For $t_0<T$,
\begin{equation}\label{eq:rotation-energy-lemma}
 \|Z_T\|_{L^2(\mathcal H)}
 \le \|Z_{t_0}\|_{L^2(\mathcal H)}
 +\sum_{t=t_0}^{T-1}\|b_t\|_{L^2(\mathcal H)}
 +\left(\sum_{t=t_0}^{T-1}\E\|\xi_t\|_{\mathcal H}^2\right)^{1/2}.
\end{equation}
\end{lemma}
\begin{proof}
Define $M_{t_0}=0$ and $M_{t+1}=\mathcal O_t(M_t+\xi_t)$. The variable $M_t$ is measurable with respect to $\mathcal F_t$. Orthogonality removes $\mathcal O_t$ before taking expectations, and conditional centering removes the cross term:
\[
 \E\|M_{t+1}\|^2=\E\|M_t\|^2+\E\|\xi_t\|^2.
\]
Thus its final squared norm is the sum of the increment energies. The difference obeys
$Z_{t+1}-M_{t+1}=\mathcal O_t(Z_t-M_t+b_t)$, so its norm is at most the initial norm plus the sum of the drift norms, pathwise. Minkowski gives~\eqref{eq:rotation-energy-lemma}. In particular, we never assert that the rotated noises $\mathcal O_t\xi_t$ are martingale differences.
\end{proof}

\subsection{Partial traces}
For an arbitrary real operator $A$ on $\Lambda^2\mathbb R^n$, extend its coefficients antisymmetrically to ordered pairs, and put
\begin{equation}\label{eq:Rdef}
 (\R A)_{ik}=\sum_j A_{ij,kj},\qquad
 \R_0A=\R A-\frac{2\Tr A}{n}I_n.
\end{equation}
The output norm is the fundamental Frobenius norm. This map is equivariant, $\R_0 I_N=0$, and
\begin{equation}\label{eq:Rbounds}
 \norm{\R_0 A}_F\le C\sqrt n\norm A_{\HS}
 \le Cn^{3/2}\norm A_{\op},\qquad
 \R(uv^*)=-2uv,\qquad
 \norm{\R_0(uv^*)}_F\le2\norm u\norm v.
\end{equation}
In the middle identity $u,v$ are viewed as fundamental skew matrices on the right. The ordered coefficient of $u$ is $\sqrt2$ times its fundamental entry; contraction gives $-2uv$. The general bound is Cauchy--Schwarz in the contracted index. Removing the scalar part is an orthogonal projection.

\begin{lemma}[Haar simple-bivector estimates]\label{lem:haar}
For a Haar-uniform unit simple bivector $w$ and any real operator $A$, including nonsymmetric $A$,
\begin{align}
 \E|\ip w{Aw}-\tau A|^2&\le Cn^{-1}\norm A_{\op}^2,\label{eq:haar-variance}\\
 \E\norm{[w,Aw]}^2&\le Cn^{-3}
 \bigl(\norm{A-\tau(A)I}_{\HS}^2+\norm{\R_0A}_F^2\bigr).
 \label{eq:haar-commutator}
\end{align}
\end{lemma}
\begin{proof}
For symmetric $T$, Gaussian Wick contraction of $z=g\wedge h=rw$, with independent standard Gaussian $g,h$, gives
\begin{equation}\label{eq:fourth-moment}
 \E\ip w{Tw}^2\le\frac C{n^4}
 \bigl((\Tr T)^2+\norm{\R T}_F^2+\norm T_{\HS}^2\bigr).
\end{equation}
The radial and angular assertions can be seen directly from Gaussian orthogonal decomposition. Write $u=g/\|g\|$ and decompose $h=\langle h,u\rangle u+h_\perp$. Conditional on $u$, the vector $h_\perp$ is a standard Gaussian in $u^\perp$, and its length is independent of its direction. Hence
\[
 z=\|g\|\,\|h_\perp\|\,
       u\wedge\frac{h_\perp}{\|h_\perp\|}.
\]
The two directions form a Haar orthonormal two-frame; they are independent of the lengths. The squared lengths have chi-squared laws with $n$ and $n-1$ degrees of freedom, independently. Their second moments are $n(n+2)$ and $(n-1)(n+1)$, giving
$\E r^4=n(n-1)(n+1)(n+2)$.
Furthermore $\E ww^*=I_N/N$: determinant-one coordinate signs kill off-diagonal coordinate-pair products, and coordinate permutations make the diagonal entries equal. Their sum is $\E\|w\|^2=1$. This also verifies the centering by $\tau(A)$ used in the lemma without any additional representation-theoretic assumption.

Explicitly,
$\ip z{Tz}=\sum_{i,j,k,l}T_{ij,kl}g_i h_j g_k h_l$.
For clarity, write the indices of the second copy as $(a,b,c,d)$. The $g$ pairings are the three pairings of $(i,k,a,c)$ and the $h$ pairings are those of $(j,l,b,d)$. Define
\[
 B(T)=\sum_{i,j,k,l}T_{ij,kl}T_{il,kj},\qquad
 |B(T)|\le4\|T\|_{\HS}^2.
\]
The factor four in this bound comes from using ordered antisymmetric pairs instead of $i<j$ and $k<l$. All nine Wick contractions are displayed here:
\[
\begin{array}{c|ccc}
 & (jl)(bd)&(jb)(ld)&(jd)(lb)\\ \hline
 (ik)(ac)&4(\Tr T)^2&\|\R T\|_F^2&\|\R T\|_F^2\\
 (ia)(kc)&\|\R T\|_F^2&4\|T\|_{\HS}^2&B(T)\\
 (ic)(ka)&\|\R T\|_F^2&B(T)&4\|T\|_{\HS}^2
\end{array}
\]
For example, the upper-left entry sets $k=i,c=a,l=j,d=b$ and produces
$(\sum_{i,j}T_{ij,ij})^2=4(\Tr T)^2$. An adjacent entry is
$\sum_{j,l}(\sum_iT_{ij,il})^2=\|\R T\|_F^2$. The middle diagonal entry identifies the two copies without crossing and gives $\sum_{i,j,k,l}T_{ij,kl}^2=4\|T\|_{\HS}^2$. The crossed identifications give $B(T)$. The table bounds the numerator by
\[
 4(\Tr T)^2+4\|\R T\|_F^2+16\|T\|_{\HS}^2.
\]
Division by $\E r^4=n(n-1)(n+1)(n+2)$ proves~\eqref{eq:fourth-moment}.

Apply it to $\operatorname{sym}(A-\tau(A)I)$ to obtain~\eqref{eq:haar-variance}, using $\norm{\R T}_{\op}\le(n-1)\norm T_{\op}$. The latter follows by Cauchy--Schwarz in
$\sum_j\ip{u\wedge e_j}{T(v\wedge e_j)}$.

For the commutator bound, subtract $\tau(A)I$ first. For an orthonormal basis $b$ of $\g$, set $K_b=\ad_b$ and
$T_b=(A^TK_b-K_bA)/2$.
Then $T_b$ is symmetric and $\ip b{[w,Aw]}=\ip w{T_bw}$.
The completeness identity
\[
 \sum_b b_{ij}b_{kl}=\tfrac12(\delta_{ik}\delta_{jl}-\delta_{il}\delta_{jk})
\]
gives $\sum_b K_b^*K_b=(n-2)I$, and hence
$\sum_b\norm{T_b}_{\HS}^2\le Cn\norm A_{\HS}^2$.
It also gives
$\sum_a[a,Aa]=(\R A-(\R A)^T)/2$, so
$\sum_b(\Tr T_b)^2\le C\norm{\R A}_F^2$.
For the remaining contractions, put
\[
 (\mathcal S_A(b))_{ik}=\sum_{j,p}b_{jp}A_{ip,kj}.
\]
An index expansion gives
\[
 \R(K_bA)=b\R A+\mathcal S_A(b),\qquad
 \R(AK_b)=(\R A)b+\mathcal S_A(b).
\]
Here is an explicit estimate for the contraction that could otherwise lose a dimension factor. For fixed $i,k$, regard $Z^{(i,k)}_{jp}=A_{ip,kj}$ as a fundamental matrix. Since the coordinate $b$ form an orthonormal basis of the skew matrices,
\[
 \sum_b\left|\sum_{j,p}b_{jp}Z^{(i,k)}_{jp}\right|^2
 =\|\operatorname{skew}Z^{(i,k)}\|_F^2
 \le\sum_{j,p}|Z^{(i,k)}_{jp}|^2.
\]
Summing $i,k$ and converting ordered pairs to the operator norm convention gives
\begin{equation}\label{eq:reshuffled-exact-bound}
 \sum_b\|\mathcal S_A(b)\|_F^2\le4\|A\|_{\HS}^2.
\end{equation}
The left-product identity follows directly from
\[
 (K_bA)_{ij,kl}=\sum_p\bigl(b_{ip}A_{pj,kl}+b_{jp}A_{ip,kl}\bigr)
\]
by setting $l=j$ and summing $j$. The right-product identity follows by the same expansion on the second pair, using antisymmetry and $b^T=-b$. It is valid without assuming $A=A^T$.
Now
\[
 \R T_b=\tfrac12\bigl((\R A)^Tb+\mathcal S_{A^T}(b)
                     -b\R A-\mathcal S_A(b)\bigr).
\]
The defining completeness relation also gives
$\sum_b b^Tb=\sum_bbb^T=(n-1)I_n/2$. Hence
\[
 \sum_b\|b\R A\|_F^2=\tfrac{n-1}{2}\|\R A\|_F^2,
 \quad
 \sum_b\|(\R A)^Tb\|_F^2=\tfrac{n-1}{2}\|\R A\|_F^2.
\]
Using~\eqref{eq:reshuffled-exact-bound}, the elementary inequality
$\|Z_1+\cdots+Z_4\|^2\le4\sum_i\|Z_i\|^2$ proves
\begin{equation}\label{eq:partial-Tb-detailed}
 \sum_b\norm{\R T_b}_F^2\le Cn\norm{\R A}_F^2+C\norm A_{\HS}^2.
\end{equation}
Sum~\eqref{eq:fourth-moment} for $T_b$ and recall the initial scalar subtraction. This proves~\eqref{eq:haar-commutator}.
\end{proof}

\section{Regularized comparison estimates}\label{sec:comparison}
Use throughout
\begin{equation}\label{eq:parameters}
 \ell=\log(en),\qquad R=\lceil n^{5/4}\rceil,\qquad
 n^{-1/100}\le\eta\le1.
\end{equation}
For sufficiently large $n$, $R=o(N)$ and $R\le N$. Constants below are absolute and uniform over $\eta$ in this interval. Parameters $z$ used as resolvent arguments satisfy $\eta\le z\le5$.

The deliberately separated exponents $1/4$ and $1/100$ have different jobs. The suffix produces representation-averaging error $\exp(-n^{1/4}/2)$, whereas approximating a relative-orientation inverse will require polynomial degree at most $Cn^{1/100}\ell^2$. Their coefficient growth is dominated by the suffix error. Removing a suffix perturbs a resolvent in at most $R$ directions, and $R/N=O(n^{-3/4})$ remains small. Finally, all inverse powers in the transport estimates are fixed powers of $\eta^{-1}$; the exponent $1/100$ leaves room for every such loss.

The following is the non-elementary spectral input, stated in the precise normalization used here. Carlen--Carvalho--Loss~\cite[Section 6]{CCL} prove that the group walk has the same gap as the corresponding sphere walk. For one discrete coordinate-plane update with a uniform angle, its $L^2(\pi_n)$ gap is
\begin{equation}\label{eq:classical-gap}
 \gamma_n=\frac{n+2}{2n(n-1)}\ge\frac1{2n}.
\end{equation}
Equivalently, for the one-step operator $\mathsf P$ on mean-zero $L^2(\pi_n)$,
$\|\mathsf P f\|_2\le(1-\gamma_n)\|f\|_2$. The operator is self-adjoint and positive: each coordinate-circle average is an orthogonal projection and $\mathsf P$ is their mean. Positivity makes the spectral-gap bound an operator-norm bound, without a second estimate for negative eigenvalues. The continuous-time normalization $n(\mathsf P-I)$ would have gap $n\gamma_n$; it is not the time convention of this paper. For every finite-dimensional unitary representation $\sigma$,
\begin{equation}\label{eq:rep-gap}
 \norm{\E\sigma(G_R)-\Pi_{\mathrm{inv}}}_{\op}
 \le e^{-R/(2n)}.
\end{equation}
Indeed, one-step averaging is an average of orthogonal circle projections. Its nontrivial irreducible summands occur in the regular representation and satisfy the group-gap bound; independence takes its $R$th power.

More explicitly, a finite-dimensional unitary representation is an orthogonal sum of irreducibles. The matrix coefficients of each irreducible span an invariant subspace of $L^2(\pi_n)$ with that representation's spectrum for convolution by the one-step law. The invariant vectors have eigenvalue one; all remaining eigenvalues lie in $[0,1-\gamma_n]$. The independent product $G_R$ has representation average equal to the $R$th power of the one-step representation average. This proves~\eqref{eq:rep-gap}. If a scalar polynomial is expanded as a sum of matrix coefficients of tensor powers of the defining representation, its expectation error is at most its total coefficient absolute sum times $e^{-R/(2n)}$. The representation dimension does not change the operator bound; dimension enters only through this explicitly counted coefficient sum.

\subsection{Cutoffs and which functions are Haarized}
Let $H_*=C_*\ell$ for a sufficiently large fixed constant. Norm variables such as
$1+\max_{t\le4N}\norm{S_t}$ or $1+\norm X+\norm Y$ for two independent blocks have tails $Cn^2e^{-u/C}$ and deterministic bounds $Cn^2$. The rational observables below contain a fixed number of covariance and resolvent factors, with inverse norm at most $\eta^{-1}\le n^{1/100}$. They, their needed squares, and all norm weights of degree at most eight are bounded by a fixed power, say $n^{100}$ for large $n$. By increasing $C_*$, contributions from the corresponding event $M>H_*$ are less than $n^{-200}$. In particular for the nonnegative observables used here,
\begin{equation}\label{eq:weight}
 \E M^j Z\le H_*^j\E Z+n^{-200},\qquad 0\le j\le8.
\end{equation}
For an $L^2$ estimate apply this to the square before taking a square root. The same rule applies to the scalar errors below, which are deterministically bounded by $2/\eta$. Finitely many sums over at most $4N$ times leave such errors negligible. We retain a common final reserve $n^{-100}$ when needed; it is always absorbed by larger displayed bounds.

For an explicit justification of the cutoff rule, if $Z\le n^b$ and $M\le Cn^2$, then
\[
 \E[M^jZ\one_{\{M>H_*\}}]
 \le C^j n^{2j+b}\PP(M>H_*)
 \le C'n^{2j+b+2}\exp(-H_*/C').
\]
Here $b$ and $j\le8$ are fixed independently of $n$ and $\eta$. In the displayed observables, a resolvent has norm at most $n^{1/100}$, a covariance has norm at most $4N$, a partial trace costs at most $Cn^{3/2}$ times operator norm, and a squared norm introduces only another fixed power. Even with the needed squares and weights, one may use a common bound $n^{100}$ for all large $n$. Choosing $C_*$ after this finite bound makes the preceding expectation smaller than $n^{-200}$. For $L^2$ truncation, apply the same argument to $Z^2$ first, reserving twice the desired exponent. This also applies to $\max_t|s_t-\sigma|$, since it is bounded by $2/\eta$. The cutoff never requires an assumption that a componentwise angle average stays in the same event.

There are two distinct Haar comparisons. If a resolvent is subjected to one common conjugation, then
$(\rho(G)S\rho(G)^*+zI)^{-1}=\rho(G)(S+zI)^{-1}\rho(G)^*$.
Diagonal variance tests are therefore degree at most eight in $G$ even when $z$ is small. Equation~\eqref{eq:rep-gap} applies directly, with a polynomial coefficient bound.

For a \emph{relative} rotation, such as $(\rho(G)X\rho(G)^*+Y+zI)^{-1}$, the dependence on $G$ is not polynomial. We justify its replacement separately.

\begin{lemma}[Haar comparison for the needed rational tests]\label{lem:rational-comparison}
Let $X,Y$ be independent of one or two $R$-step separator rotations, with $X,Y\succeq0$, operator norms at most $4N$, and the norm tails above. In each of the following tests, replacing either independent separator by Haar changes the expectation by at most $n^{-100}$:
\[
 \norm{\R_0 T}_F^2,\quad \norm{\R_0(YT)}_F^2,\quad
 Q(T),\quad Q(YT),\qquad
 Q(A)=\frac1N\sum_a\norm{L_aAa}^2.
\]
Here $S$ is the sum of the two covariances with their stated independent conjugations, $T=(S+zI)^{-1}$, and $Y$ in $YT$ denotes the corresponding conjugated colored covariance. The assertion is uniform in $z\in[\eta,5]$. No cutoff depending on the separator is included in a polynomial test.
\end{lemma}
\begin{proof}
First restrict the \emph{coefficient} matrices to $\norm X+\norm Y\le H_*$. This event is independent of the separators and invariant under any of their conjugations. Outside it, the original rational tests and their Haar counterparts contribute less than $n^{-200}$, by~\eqref{eq:weight} and their deterministic polynomial bounds.

On this event $0\preceq S\preceq H_*I$ for every possible separator. With $H_0=2H_*+6$ define
\[
 p_J(S)=H_0^{-1}\sum_{j=0}^J
          \bigl(I-(S+zI)/H_0\bigr)^j,
 \qquad J=\lceil1000H_0\eta^{-1}\ell\rceil.
\]
The geometric remainder gives, uniformly in all separator values,
\[
 \norm{(S+zI)^{-1}-p_J(S)}_{\op}
 \le\eta^{-1}e^{-(J+1)\eta/H_0}\le\eta^{-1}n^{-1000}.
\]
Replacing $T$ by $p_J(S)$, and $YT$ by $Yp_J(S)$, changes any displayed test by less than $n^{-500}$; use~\eqref{eq:Rbounds}, the at-most-polynomial dimension factors, and $\norm Y\le H_*$.

For example, if $T'$ is an approximant with $\|T-T'\|_{\op}\le e_J$, then
\[
 \big|\|\R_0T\|_F^2-\|\R_0T'\|_F^2\big|
 \le C n^3 e_J(2\eta^{-1}+e_J).
\]
For $YT$, multiply the right side by $\|Y\|^2$. For the test $Q$, the seminorm $Q(A)^{1/2}$ is at most $\|A\|_{\op}$, because $\|L_a\|\le1$ and the coordinate vectors have norm one. Its squared difference is therefore bounded by
$e_J(2\eta^{-1}+e_J)$, again with a factor $\|Y\|^2$ in the mixed case. Equation~\eqref{eq:Rbounds} and $H_*=O(\ell)$ make all these errors smaller than $n^{-500}$ with the stated choice of $J$.

An entry of $\rho(U)$ is a two-by-two minor of $U$. Expanding the polynomial tests gives degree $O(J)$ in the one or two separators and total coefficient absolute sum at most $\exp(CJ\log n)$. This bound follows by expanding each matrix product into its $N\le n^2$-valued indices; the coefficients of $p_J$ and the $O(J)$ factors contribute at most an additional $\exp(CJ\log H_*)$, which is included.

One can check the count entry by entry. An entry of $\rho(G)X\rho(G)^*$ is a sum over two Lie-algebra indices, each with at most $n^2$ values; each $\rho(G)$ entry is a sum of two degree-two monomials. A product of $J$ such matrices has $J-1$ additional summed Lie-algebra indices. Traces and partial traces add only a fixed number of summed indices, and a squared test doubles the number of factors. Thus both the number of choices and the coefficient mass are bounded by $(C n^c)^{C(J+1)}$ for absolute exponents. Coefficients containing $z\in[\eta,5]$ and $H_0^{-1}$ are bounded by a fixed power of $H_0$ per factor, also covered by this envelope. In particular the logarithm grows linearly, not quadratically, in $J\log n$.
 Each monomial is a matrix coefficient of a tensor power of the defining representation. By~\eqref{eq:rep-gap}, the expectation changes by at most
\[
 \exp(CJ\log n-R/(2n)).
\]
Here $J\le C n^{1/100}\ell^2$, whereas $R/n\ge n^{1/4}$. The last display is less than $n^{-500}$ for all sufficiently large $n$, uniformly in the parameter range. Replace two separators successively. Integrating the conditional estimate over the good coefficient event and restoring its complement proves the assertion. Polynomial approximants on the bad coefficient event are never used.
\end{proof}

\subsection{Diagonal homogenization of a resolvent}
For any operator $A$, let
\[
 V(A)=\frac1N\sum_a|\ip a{Aa}-\tau(A)|^2.
\]
\begin{lemma}\label{lem:diag-resolvent}
For a total colored covariance $S_t$, $R\le t\le4N$, and deterministic $z\in[\eta,5]$,
\begin{equation}\label{eq:diag-resolvent}
 \E V((S_t+zI)^{-1})\le C n^{-3/4}\eta^{-2}.
\end{equation}
With a norm weight $M^j$, $j\le8$, the right side may be multiplied by $C_j\ell^j$.
\end{lemma}
\begin{proof}
Split the last $R$ actual rotations:
$S_t=\rho(G)S_{t-R}\rho(G)^*+B$, where $B\succeq0$ and $\rank B\le R$.
Put $T=(S_t+zI)^{-1}$ and
$T^{\mathrm o}=\rho(G)(S_{t-R}+zI)^{-1}\rho(G)^*$.
Then $0\preceq T^{\mathrm o}-T\preceq z^{-1}I$ and $\rank(T^{\mathrm o}-T)\le R$. Thus
\[
 \norm{T-T^{\mathrm o}}_{\HS}^2\le R\eta^{-2}.
\]
Conditional on the old covariance, $V(T^{\mathrm o})$ is a fixed-degree test in the fresh $G$. Common-conjugation Haar comparison and~\eqref{eq:haar-variance} bound its expectation by $C/(n\eta^2)+n^{-200}$. Since $V(E)\le\norm E_{\HS}^2/N$,
\[
 \E V(T)\le C\eta^{-2}(n^{-1}+R/N)+n^{-200}
 \le C n^{-3/4}\eta^{-2}.
\]
The weights follow from~\eqref{eq:weight}; their complement is negligible. The recent covariance and $G$ need not be independent.
\end{proof}

\subsection{A scalar stability inequality with cumulative noise}
\begin{lemma}[Pathwise stability and a martingale forcing bound]\label{lem:scalar-Gronwall}
Suppose $e_0=0$ and
\[
 e_{t+1}=(1+A_t)e_t+b_t+\xi_t,
 \qquad 0\le A_t\le\alpha_t,
\]
where the $\alpha_t$ are deterministic nonnegative numbers. Define
$F_u=\sum_{t<u}(b_t+\xi_t)$. Then
\begin{equation}\label{eq:Gronwall-pathwise}
 \max_{u\le m}|e_u|
 \le\left(\prod_{t<m}(1+\alpha_t)\right)\max_{u\le m}|F_u|.
\end{equation}
If $\xi_t$ is a square-integrable martingale difference, then
\begin{equation}\label{eq:Gronwall-L2}
 \left\|\max_{u\le m}|e_u|\right\|_2
 \le\left(\prod_{t<m}(1+\alpha_t)\right)
 \left[\sum_{t<m}\|b_t\|_2+
 2\left(\sum_{t<m}\E\xi_t^2\right)^{1/2}\right].
\end{equation}
No independence between $A_t$ and the preceding noises is required.
\end{lemma}
\begin{proof}
Summing the recurrence gives
$e_u=F_u+\sum_{t<u}A_te_t$. Fix one sample path and put
$f=\max_{u\le m}|F_u|$. Induction in $u$ yields
\[
 |e_u|\le f\prod_{t<u}(1+\alpha_t),
\]
because $1+\sum_{t<u}\alpha_t\prod_{j<t}(1+\alpha_j)=\prod_{t<u}(1+\alpha_t)$. This proves~\eqref{eq:Gronwall-pathwise}. The martingale $M_u=\sum_{t<u}\xi_t$ has
$\E M_m^2=\sum_t\E\xi_t^2$, by conditional centering. Its $L^2$ maximal inequality is
$\|\max_u|M_u|\|_2\le2\|M_m\|_2$. One proof applies the stopping-time bound
$\lambda\PP(M^*>\lambda)\le\E[|M_m|\one_{\{M^*>\lambda\}}]$
to the nonnegative submartingale $|M_u|$, integrates in $\lambda$, and uses Cauchy--Schwarz; it gives $\E(M^*)^2\le2\|M_m\|_2\|M^*\|_2$. The stopping-time bound follows by stopping at the first crossing and conditioning $|M_m|$ on that time. Adding the cumulative drift by Minkowski proves~\eqref{eq:Gronwall-L2}.
\end{proof}

\section{Following a deterministic characteristic}\label{sec:characteristic}
Fix a deterministic colored schedule of length $m\le4N$. Let $v_t=\sum_{j<t}\varepsilon_j$ be the number of retained increments by time $t$ and set $c=v_m/N\in[0,4]$. Fix the final regularization $\eta$ in~\eqref{eq:parameters}. Let $z_0>0$ be the positive solution of
\begin{equation}\label{eq:z0}
 \eta=z_0-\frac{c}{1+z_0^{-1}},\qquad
 z_0=\frac{c+\eta-1+\sqrt{(c+\eta-1)^2+4\eta}}2.
\end{equation}
Put
\begin{equation}\label{eq:path}
 \sigma=z_0^{-1},\quad d_0=\frac1{N(1+\sigma)},\quad
 z_t=z_0-d_0v_t,\quad T_t=(S_t+z_tI)^{-1},\quad s_t=\tau(T_t).
\end{equation}
Then $z_m=\eta$, $\eta\le z_t\le z_0\le5$, and $d_0\le1/N$. The path $z_t$ is deterministic, so the preceding conditional Haar comparisons remain valid. It stays fixed at inactive steps.

To see how the path was chosen, the predicted trace of the resolvent would remain equal to $\sigma$ if one retained rank-one increment were offset by decreasing $z$ by $1/[N(1+\sigma)]$. Start from $S_0=0$, where $s_0=1/z_0=\sigma$, and require that the final argument be $\eta$. This gives exactly~\eqref{eq:z0}. Multiplying that equation by $z_0+1$ yields
\[
 z_0^2+(1-c-\eta)z_0-\eta=0.
\]
Its constant term is negative, so it has a unique positive root. The relation
$z_0-\eta=c z_0/(z_0+1)$ shows $z_0\ge\eta$ and $z_0\le c+\eta\le5$. For an inactive step, neither the retained count nor $z_t$ changes, and conjugation does not change $s_t$.

The path need not follow the random trace exactly. Instead the discrepancy between its deterministic velocity and the random inverse drift is proportional to $s_t-\sigma$. This is the cancellation in~\eqref{eq:error-rec}; the proof below controls that discrepancy without an adapted regularization parameter.

\begin{proposition}[Scalar stability]\label{prop:scalar}
Uniformly over these schedules,
\begin{equation}\label{eq:scalar-stability}
 \left\|\max_{t\le m}|s_t-\sigma|\right\|_{L^2}
 \le C n^{-3/8}\eta^{-4}.
\end{equation}
In particular, for an uncolored $2N$-step core,
\begin{equation}\label{eq:scalar-final}
 \E\tau(D_{2N}+\eta I)^{-1}\le2
\end{equation}
for sufficiently large $n$, uniformly in~\eqref{eq:parameters}.
\end{proposition}
\begin{proof}
At an active step, pull back the common rotation and put $T=T_t$, $z=z_t$, $d=d_0$. The exact update is
\begin{equation}\label{eq:exact-resolvent}
 T_{t+1}^{\mathrm{pb}}
 =\widehat T-\frac{(\widehat T a)(\widehat T a)^*}
                       {1+\ip a{\widehat T a}},\qquad
 \widehat T=(S_t+(z-d)I)^{-1}.
\end{equation}
The full interval from $z-d$ to $z$ lies above $\eta$. Taylor expansion of the first term gives
$\widehat T=T+dT^2+E$ with $\norm E_{\op}\le d^2\eta^{-3}$.

For the rank-one term write
\[
 t_a=T_za,\quad x=\langle a,T_za\rangle,\quad y=\|t_a\|^2,
 \qquad Q_z=\frac{t_at_a^*}{1+x}.
\]
Since $\partial_z t_a=-T_zt_a$ and $\partial_zx=-y$,
\begin{equation}\label{eq:rank-derivative-expanded}
 \partial_zQ_z
 =-\frac{(T_zt_a)t_a^*+t_a(T_zt_a)^*}{1+x}
   +\frac{y\,t_at_a^*}{(1+x)^2}.
\end{equation}
For a rank-one operator $uv^*$ the nuclear norm is $\|u\|\|v\|$. Also
$\|T_zt_a\|\le\eta^{-1}\|t_a\|$ and $y\le x/\eta$. Thus
\[
 \|\partial_zQ_z\|_{S_1}
 \le\frac{2\eta^{-1}y}{1+x}+\frac{y^2}{(1+x)^2}
 \le3\eta^{-2}.
\]
Integrating over the argument interval of length $d$ bounds the rank-one shift error in nuclear norm by $3d\eta^{-2}$. The normalized trace of this error is at most $3d/(N\eta^2)$, while the full-matrix Taylor remainder has normalized trace at most $d^2\eta^{-3}$. Since $d\le1/N$ and $\eta\le1$, their sum is at most $C/(N^2\eta^3)$.

Consequently, with $x_a=\ip a{Ta}$ and $y_a=\ip a{T^2a}$,
\begin{equation}\label{eq:trace-increment}
 \Delta s_t=d\tau(T^2)-\frac1N\frac{y_a}{1+x_a}+\epsilon_{t,a},
 \qquad |\epsilon_{t,a}|\le\frac C{N^2\eta^3}.
\end{equation}
Write $s=s_t$ and $u=\tau(T^2)$. The denominator error obeys
\[
 \left|\frac1N\sum_a\frac{y_a}{1+x_a}-\frac u{1+s}\right|
 \le\eta^{-2}\sqrt{V(T)}.
\]
Indeed the function $x\mapsto(1+x)^{-1}$ is one-Lipschitz on $[0,\infty)$ and $y_a\le\eta^{-2}$.

More explicitly, subtract the two expressions before estimating:
\[
 \frac1N\sum_a\frac{y_a}{1+x_a}-\frac{u}{1+s}
 =\frac1N\sum_a y_a\left(\frac1{1+x_a}-\frac1{1+s}\right),
\]
since $N^{-1}\sum_a y_a=\tau(T^2)=u$. Its absolute value is at most
$\eta^{-2}N^{-1}\sum_a|x_a-s|\le\eta^{-2}\sqrt{V(T)}$.
Consequently no homogenization of the coordinate diagonals of $T^2$ is needed; the argument uses only $V(T)$.
 With $e_t=s_t-\sigma$, conditional averaging gives
\begin{equation}\label{eq:error-rec}
 e_{t+1}=e_t+A_te_t+b_t+\xi_t,\qquad
 A_t=\frac{d_0\tau(T_t^2)}{1+s_t},\qquad
 \E(\xi_t\mid\mathcal F_t)=0.
\end{equation}
At inactive steps every term except $e_t$ is zero.

Here the definitions are exact: $\xi_t=\Delta s_t-\E(\Delta s_t\mid\mathcal F_t)$, and
$b_t=\E(\Delta s_t\mid\mathcal F_t)-A_t(s_t-\sigma)$. Thus the Taylor remainders and denominator error are included in $b_t$; they are not silently dropped from the recurrence. The number $A_t$ is $\mathcal F_t$-measurable and nonnegative.
 The drift cancellation used here is exactly
\[
 d_0-\frac1{N(1+s_t)}=\frac{d_0(s_t-\sigma)}{1+s_t}.
\]
For active $t\ge R$, Lemma~\ref{lem:diag-resolvent} gives
\[
 \norm{b_t}_2\le C N^{-1}n^{-3/8}\eta^{-3}
                  +C N^{-2}\eta^{-3}.
\]
For $t<R$, the crude bound is $C/(N\eta^3)$. Hence
\begin{equation}\label{eq:b-sum}
 \sum_{t<m}\norm{b_t}_2\le C n^{-3/8}\eta^{-3},
\end{equation}
since $R/N=O(n^{-3/4})$.

The exact trace increment in~\eqref{eq:exact-resolvent} has absolute value at most $C/(N\eta^2)$. The change of parameter contributes at most $d_0\eta^{-2}$ and the rank-one subtraction at most $1/(N\eta)$. Thus Doob's inequality bounds the running martingale sum of $\xi_t$ in $L^2$ by $C/(n\eta^2)$.

Most importantly,
\begin{equation}\label{eq:stability-coefficient}
 0\le A_t\le\frac{d_0}{z_t},
\end{equation}
because $\tau(T_t^2)\le s_t/z_t$. At an active step,
\[
 1+d_0/z_t\le z_t/(z_t-d_0)=z_t/z_{t+1}.
\]

Set $\alpha_t=\varepsilon_t d_0/z_t$. The deterministic product is bounded by
\begin{equation}\label{eq:telescoping-stability-product}
 \prod_{t<m}(1+\alpha_t)
 \le\prod_{t:\varepsilon_t=1}\frac{z_t}{z_{t+1}}
 =\frac{z_0}{\eta}\le\frac5\eta.
\end{equation}
Lemma~\ref{lem:scalar-Gronwall}, applied to the exact error recurrence, therefore gives
\[
 \norm{\max_{t\le m}|e_t|}_2
 \le\frac5\eta\left(\frac C{n\eta^2}
                       +\sum_{t<m}\norm{b_t}_2\right)
 \le C n^{-3/8}\eta^{-4}.
\]
This proves~\eqref{eq:scalar-stability}. For $c=2$, equation~\eqref{eq:z0} has $z_0\ge1$, so $\sigma\le1$. Also $n^{-3/8}\eta^{-4}\le n^{-3/8+4/100}\to0$, uniformly in the allowed range. This gives~\eqref{eq:scalar-final}.

The statement for a core is consequently uniform over all allowed choices of the terminal argument, not merely true along one subsequence. Notice also the direction of the implication: we deduce a bounded average regularized inverse, not an operator-norm bound on that inverse. A small number of exceptionally small covariance eigenvalues is not excluded at this stage.

\end{proof}

\section{Two matrix-valued resolvents, without a degree hierarchy}\label{sec:matrix}
Continue to use the deterministic path~\eqref{eq:path}, and set
\[
 U_t=Y_tT_t,\qquad A_t^{\mathrm R}=\R_0T_t,\qquad B_t^{\mathrm R}=\R_0U_t.
\]
The mixed resolvent $U_t$ need not be symmetric. This is why the nonsymmetric statements in Lemma~\ref{lem:haar} are needed.

\begin{proposition}\label{prop:partial-resolvent}
Uniformly for every deterministic schedule of length at most $4N$ and every allowed terminal regularization $\eta$,
\begin{equation}\label{eq:partial-resolvent}
 \sup_{t\le m}\norm{\R_0T_t}_{L^2(F)}
 \le Cn^{5/4}\eta^{-6},\qquad
 \sup_{t\le m}\norm{\R_0(Y_tT_t)}_{L^2(F)}
 \le Cn^{5/4}\ell\eta^{-6}.
\end{equation}
In particular, for a colored covariance at any terminal time and fixed argument $\eta$, both squared partial traces have expectation at most
$Cn^{5/2}\ell^2\eta^{-12}$.
\end{proposition}
\begin{proof}
We display the error bounds in the two recurrences. At an active step let $T=T_t$, $Y=Y_t$, $U=YT$, $d=d_0$, and $\chi=\chi_t$. Pull back the common rotation. Besides~\eqref{eq:exact-resolvent}, the exact mixed update is
\begin{equation}\label{eq:mixed-exact}
 U_{t+1}^{\mathrm{pb}}
 =Y\widehat T+
   \frac{(\chi a-Y\widehat T a)(\widehat T a)^*}
        {1+\ip a{\widehat T a}}.
\end{equation}
Define
\[
 Q_T(a)=\frac{(Ta)(Ta)^*}{1+\ip a{Ta}},\qquad
 Q_U(a)=\frac{(\chi a-Ua)(Ta)^*}{1+\ip a{Ta}}.
\]

The mixed identity can be checked without any commutation assumption on $Y$ and $T$. If $\widehat x=\langle a,\widehat Ta\rangle$, then
\begin{align*}
 (Y+\chi aa^*)\left(\widehat T-
 \frac{(\widehat Ta)(\widehat Ta)^*}{1+\widehat x}\right)
 &=Y\widehat T+\chi a(\widehat Ta)^*
   -\frac{(Y\widehat Ta+\chi\widehat x a)(\widehat Ta)^*}{1+\widehat x}\\
 &=Y\widehat T+
   \frac{(\chi a-Y\widehat Ta)(\widehat Ta)^*}{1+\widehat x}.
\end{align*}
In particular, the ordered product in the leading term is $UT=YT^2$, not $TYT$.

Expanding the deterministic parameter shift gives
\begin{align}
 T_{t+1}^{\mathrm{pb}}&=T+dT^2-Q_T(a)+E_{T,a},\label{eq:Texp}\\
 U_{t+1}^{\mathrm{pb}}&=U+dUT+Q_U(a)+E_{U,a}.\label{eq:Uexp}
\end{align}
The remainders satisfy the following sufficient bounds:
\begin{align}
 \norm{\R_0E_{T,a}}_F&\le C N^{-1}\eta^{-4},\label{eq:ET}\\
 \norm{\R_0E_{U,a}}_F&\le C N^{-1}(1+\norm Y_{\op})\eta^{-4}.
 \label{eq:EU}
\end{align}
Here and below they are zero at inactive steps. To verify these bounds, the full-matrix Taylor remainder of $\widehat T$ has norm at most $d^2\eta^{-3}$; apply~\eqref{eq:Rbounds}, multiplying by $\norm Y$ in the mixed case. The rank-one terms and their derivatives with respect to the argument have dimension-free partial-trace bounds by~\eqref{eq:Rbounds}. For $Q_T$, the derivative is a sum of three rank-one terms with total product norm at most $C\eta^{-2}$, as in the preceding proof. For $Q_U$, differentiating its two numerator vectors and denominator bounds that total by $C(1+\norm Y)\eta^{-4}$. Integrate over an interval of length $d\le1/N$. Finally $n^{3/2}/N$ is bounded, so $d^2 n^{3/2}\eta^{-3}\le C N^{-1}\eta^{-4}$. This proves~\eqref{eq:ET}--\eqref{eq:EU}.

A detailed check for the mixed rank-one derivative is useful. At a variable argument $z$, write
\[
 Q_{U,z}(a)=\chi\frac{a(T_za)^*}{1+x}-YQ_z,
 \qquad x=\langle a,T_za\rangle.
\]
The derivative of the first term is a sum of two rank-one operators:
\[
 -\chi\frac{a(T_z^2a)^*}{1+x}
 +\chi\frac{y\,a(T_za)^*}{(1+x)^2},\qquad y=\|T_za\|^2.
\]
Its sum of product norms is at most $C\eta^{-3}$ for $\eta\le1$. Multiplication of the three rank-one derivatives in~\eqref{eq:rank-derivative-expanded} by $Y$ costs at most $\|Y\|$. Applying the dimension-free partial-trace bound separately to these five terms gives
\[
 \|\partial_z\R_0Q_{U,z}(a)\|_F
 \le C(1+\|Y\|)\eta^{-4}.
\]
For the deterministic Taylor remainders the general partial-trace inequality gives
$Cn^{3/2}d^2\eta^{-3}$, or this times $\|Y\|$. Integrating the rank-one derivatives over length $d\le N^{-1}$ proves the two claimed error bounds with an absolute constant. This use of rank-one norms, rather than the general partial-trace inequality on the random update, preserves the needed power of $n$.

Write $s=\tau T$ and
\[
 \delta_a=\frac1{1+\ip a{Ta}}-\frac1{1+s}.
\]
Then $N^{-1}\sum_a\delta_a^2\le V(T)$. The exact conditional means of the rank-one terms at fixed $T$ are
\begin{align}
 \E_aQ_T(a)&=\frac{T^2}{N(1+s)}+
                \frac1N T\operatorname{diag}(\delta_a)T,
                \label{eq:Tmean}\\
 \E_aQ_U(a)&=\frac{\chi T-UT}{N(1+s)}+
                \frac1N(\chi I-U)\operatorname{diag}(\delta_a)T.
                \label{eq:Umean}
\end{align}
By the general partial-trace bound, their last terms obey

For example, $\|\operatorname{diag}(\delta_a)\|_{\HS}=\sqrt{N\,N^{-1}\sum_a\delta_a^2}$. Hence
\[
 \frac{C\sqrt n}{N}\|T\|_{\op}^2
       \|\operatorname{diag}(\delta_a)\|_{\HS}
 \le C\sqrt{\frac nN}\eta^{-2}\sqrt{V(T)}.
\]
For the mixed term, use $\|\chi I-U\|\le1+\|Y\|/\eta\le(1+\|Y\|)/\eta$. These are the factors behind the following bounds:

\begin{align}
 \left\|\R_0\left[N^{-1}T\operatorname{diag}(\delta_a)T\right]\right\|_F
 &\le C\sqrt{n/N}\,\eta^{-2}\sqrt{V(T)},\label{eq:den-T}\\
 \left\|\R_0\left[N^{-1}(\chi I-U)\operatorname{diag}(\delta_a)T\right]\right\|_F
 &\le C\sqrt{n/N}(1+\norm Y)\eta^{-2}\sqrt{V(T)}.
 \label{eq:den-U}
\end{align}
For $t\ge R$, Lemma~\ref{lem:diag-resolvent} and its weighted form make their $L^2$ norms at most
\begin{equation}\label{eq:den-L2}
 Cn^{-7/8}\eta^{-3},\qquad C\ell n^{-7/8}\eta^{-3},
\end{equation}
respectively.

The leading matrix coefficient after combining the shift and the mean update is
\begin{equation}\label{eq:gcoef}
 g_t=d_0-\frac1{N(1+s_t)}
     =\frac{d_0(s_t-\sigma)}{1+s_t}.
\end{equation}
Therefore the unrotated mean increment of $\R_0T$ is
$g_t\R_0(T_t^2)$, apart from the denominator and Taylor errors above. Since
$\norm{\R_0(T_t^2)}_F\le Cn^{3/2}\eta^{-2}$,
Proposition~\ref{prop:scalar} gives
\begin{equation}\label{eq:leading-T}
 \norm{g_t\R_0(T_t^2)}_2
 \le \frac C N n^{9/8}\eta^{-6}.
\end{equation}
For the mixed resolvent, the corresponding increment is
\begin{equation}\label{eq:leading-U-form}
 g_t\R_0(U_tT_t)+\frac{\chi_t}{N(1+s_t)}\R_0T_t.
\end{equation}
The first term has $L^2$ norm at most
$C N^{-1}n^{9/8}\ell\eta^{-6}$.
Indeed $\norm{\R_0(U_tT_t)}\le Cn^{3/2}\norm{Y_t}\eta^{-2}$, and the cutoff rule bounds
$\norm{\norm{Y_t}|s_t-\sigma|}_2$ by
$C\ell n^{-3/8}\eta^{-4}$ plus a negligible error.
The second term is at most $N^{-1}\norm{\R_0T_t}_2$ in $L^2$.

We explain how to sum these increments. For either partial trace, its exact recursion has the form
\[
 Z_{t+1}=V_t(Z_t+\beta_t+\xi_t)V_t^T,
 \qquad \E(\xi_t\mid\mathcal F_t)=0,
\]
where $V_t$ is the fresh fundamental rotation and $\beta_t$ is the conditional mean of the pulled-back increment. The rotation may depend on the centered increment. Define
$M_{t+1}=V_t(M_t+\xi_t)V_t^T$, starting from zero at time $R$. Norm invariance removes $V_t$ before taking expectations, so
\[
 \E\norm{M_T}_F^2=\sum_{t=R}^{T-1}\E\norm{\xi_t}_F^2.
\]
The difference $Z_T-M_T$ has norm at most
$\norm{Z_R}+\sum_{t\ge R}\norm{\beta_t}$.
This argument needs no assertion that the rotated increments are martingale differences.

The exact random parts in~\eqref{eq:exact-resolvent} and~\eqref{eq:mixed-exact}, together with the rank-one bound, show that the centered-increment $L^2$ norms are at most
$C\eta^{-2}$ and $C\ell\eta^{-2}$ respectively. Their accumulated noise norms over $4N$ steps are therefore at most $Cn\eta^{-2}$ and $Cn\ell\eta^{-2}$.
For the first $R$ steps, simply sum the exact increments. The deterministic shift contributes at most $C d_0n^{3/2}\eta^{-2}$ per step, or this times $\norm{Y_t}$ in the mixed case, and the rank-one term contributes at most $C\eta^{-2}$ or $C(1+\norm{Y_t})\eta^{-2}$. Thus their initial $L^2$ norms are bounded by
\begin{equation}\label{eq:initial-matrix}
 CR\eta^{-2},\qquad CR\ell\eta^{-2}.
\end{equation}
This also covers any terminal time less than $R$.

For reference, the complete $L^2$ ledger for the unweighted partial trace is
\begin{center}
\small
\begin{tabular}{p{0.38\textwidth}p{0.23\textwidth}p{0.28\textwidth}}
\toprule
Contribution & Per active step, after time $R$ & Accumulated through $4N$\\
\midrule
Taylor remainder & $CN^{-1}\eta^{-4}$ & $C\eta^{-4}$\\
Denominator error & $Cn^{-7/8}\eta^{-3}$ & $Cn^{9/8}\eta^{-3}$\\
Scalar-error drift & $CN^{-1}n^{9/8}\eta^{-6}$ & $Cn^{9/8}\eta^{-6}$\\
Centered noise & $C\eta^{-2}$ & $Cn\eta^{-2}$\\
Initial interval & --- & $Cn^{5/4}\eta^{-2}$\\
\bottomrule
\end{tabular}
\end{center}
Only the noise row is summed quadratically; the three drift rows use Minkowski. Every row for the mixed partial trace has at most an additional factor $C\ell$, except its source $\chi\R_0T/[N(1+s)]$, which sums to at most $4\sup_t\|\R_0T_t\|_2$. Lemma~\ref{lem:rotated-energy} is the precise accumulation rule.

For $\R_0T$, summing~\eqref{eq:ET}, \eqref{eq:den-L2}, and~\eqref{eq:leading-T}, with the noise and initial bounds, gives
\[
 \norm{\R_0T_T}_2\le C\left(
 n^{5/4}\eta^{-2}+n\eta^{-2}
 +n^{9/8}\eta^{-6}+n^{9/8}\eta^{-3}+\eta^{-4}\right)
 \le Cn^{5/4}\eta^{-6}.
\]
The bound is uniform in $T\le m$. Use it in the second term of~\eqref{eq:leading-U-form}. Its accumulation is at most four times that uniform bound. All the other mixed estimates are the preceding ones with at most a factor $C\ell$. Consequently
$\norm{\R_0U_T}_2\le Cn^{5/4}\ell\eta^{-6}$.
This proves~\eqref{eq:partial-resolvent}. All discarded tails are smaller than the displayed bounds uniformly for $\eta\ge n^{-1/100}$.
\end{proof}

\begin{remark}
The proof does not bound a sequence $\R_0T^k$ recursively. The only higher resolvent products are $T^2$ and $YT^2$ multiplied by the already-small scalar error $s_t-\sigma$. They are estimated deterministically. The source in the mixed equation contains $\R_0T$, which has already been controlled. This is the closure that replaces the polynomial degree induction.
\end{remark}

\section{Two independent covariance blocks}\label{sec:two-block}
For independent Kac covariances $X,Y$ whose lengths sum to at most $2N$, define
\[
 T=(X+Y+\eta I)^{-1},\qquad U=YT,\qquad
 Q(A)=\frac1N\sum_a\norm{L_aAa}^2.
\]
Each covariance may be left- or right-trivialized; their marginal distributions coincide.

\begin{proposition}\label{prop:two-block}
For $0\le j\le4$, uniformly in the lengths and allowed $\eta$,
\begin{equation}\label{eq:two-block-resolvent}
 \E(1+\norm X+\norm Y)^j\bigl(Q(T)+Q(U)\bigr)
 \le C_j n^{-1/2}\ell^{j+4}\eta^{-12}.
\end{equation}
\end{proposition}
\begin{proof}
We first establish a relative-Haar version of the partial-trace estimate. Run an auxiliary colored process consisting of the $X$ core, $R$ inactive rotations $G_R$, and the $Y$ core. Conjugate its terminal pair back by the endpoint of the last core. The pair becomes exactly
\[
 (\rho(G_R)X\rho(G_R)^*,Y),
\]
where $X,G_R,Y$ are independent and $Y$ is left-trivialized. The auxiliary trajectory has at most $3N$ actual rotations, for large $n$. Apply Proposition~\ref{prop:partial-resolvent} at terminal regularization $\eta$; the centered partial-trace norms are invariant under the common final conjugation. Lemma~\ref{lem:rational-comparison} then gives
\begin{equation}\label{eq:relative-partial}
 \E\norm{\R_0(\rho(G)X\rho(G)^*+Y+\eta I)^{-1}}_F^2
 +\E\norm{\R_0\!\left[Y(\rho(G)X\rho(G)^*+Y+\eta I)^{-1}\right]}_F^2
 \le C n^{5/2}\ell^2\eta^{-12}
\end{equation}
for independent Haar $G$. The same bound holds with independent Haar conjugations of both arguments: remove the common rotation of the second argument. Its relative rotation against the first is an independent Haar variable. These auxiliary rotations are used only to prove an expectation estimate; they are not added to the coupled walk.

The independence assertion here is exact and deserves emphasis. Let $H$ be the endpoint of the last core, $D_Y$ its right covariance, and $D_X$ the covariance of the first core. Before the final conjugation, the two colored matrices are
\[
 \rho(HG_R)D_X\rho(HG_R)^*,\qquad D_Y.
\]
Conjugating both by $\rho(H)^*$ gives $\rho(G_R)D_X\rho(G_R)^*$ and $\rho(H)^*D_Y\rho(H)$, the latter being the left covariance of the last core. Although $H$ and $D_Y$ are dependent, neither remains in the first expression. The two cores and separator were generated independently, so $D_X,G_R,$ and that left covariance are mutually independent. This is why we can condition on the coefficient matrices in the rational Haar comparison. We never split a core's endpoint from its own covariance by an independence assumption.

For the actual pair $X,Y$, split the last $R$ steps of each right covariance:
\[
 X=X^{\mathrm o}+B_X,\qquad Y=Y^{\mathrm o}+B_Y,
 \qquad B_X,B_Y\succeq0,\quad \Tr B_X+\Tr B_Y\le2R.
\]
If a block is shorter than $R$, put its whole covariance into the recent part and set its old part to zero. Each nonzero old covariance is an independent $R$-step conjugation of an independent earlier covariance. Set
\[
 T^{\mathrm o}=(X^{\mathrm o}+Y^{\mathrm o}+\eta I)^{-1},
 \qquad U^{\mathrm o}=Y^{\mathrm o}T^{\mathrm o}.
\]
The difference $T^{\mathrm o}-T$ is positive semidefinite, of rank at most $2R$, and norm at most $\eta^{-1}$. Therefore
\begin{equation}\label{eq:resolvent-recent}
 \norm{T-T^{\mathrm o}}_{\HS}^2\le2R\eta^{-2}.
\end{equation}
For the mixed resolvent,
\[
 U-U^{\mathrm o}=B_YT+Y^{\mathrm o}(T-T^{\mathrm o}),
\]
so the norm tail, $\norm{B_Y}_{\HS}^2\le\norm Y\Tr B_Y$, and~\eqref{eq:resolvent-recent} give
\begin{equation}\label{eq:mixed-recent}
 \E\norm{U-U^{\mathrm o}}_{\HS}^2
 \le C R\ell^2\eta^{-2}.
\end{equation}
No independence between recent covariances and their endpoint rotations is needed in these bounds.

Indeed, the two terms in the mixed difference satisfy, pathwise,
\[
 \|B_YT\|_{\HS}^2\le\eta^{-2}\|B_Y\|_{\HS}^2
 \le2R\eta^{-2}\|Y\|,
\]
\[
 \|Y^{\mathrm o}(T-T^{\mathrm o})\|_{\HS}^2
 \le\|Y\|^2\,2R\eta^{-2}.
\]
The inequality $\|A+B\|^2\le2\|A\|^2+2\|B\|^2$ and the fixed operator-norm moments from Lemma~\ref{lem:norm-tail} prove~\eqref{eq:mixed-recent}. The same calculation works if a core is shorter than $R$, because its complete covariance then has trace equal to its length, at most $R$.

For any $E$, $Q(E)\le\norm E_{\HS}^2/N$ since $\norm{L_a}\le1$. The recent terms in~\eqref{eq:resolvent-recent}--\eqref{eq:mixed-recent} thus contribute at most
$C n^{-3/4}\ell^2\eta^{-2}$ to the unweighted claim. Haarize the independent old separators using Lemma~\ref{lem:rational-comparison}. The resulting pair $(\widetilde X,\widetilde Y)$ is invariant under common Haar conjugation. This remains true when one old block is zero. Let $\widetilde T=(\widetilde X+\widetilde Y+\eta I)^{-1}$ and $\widetilde U=\widetilde Y\widetilde T$.

Introduce a common independent Haar rotation and rotate the vector expression back. The coordinate generator becomes a Haar simple bivector $w$.

To justify this without an independence shortcut, condition first on one realization of the pair $(\widetilde X,\widetilde Y)$ and then apply an additional independent Haar rotation $V$ to both. Their unconditional law is unchanged by this operation. For either resulting operator $A$, rotate the expression $L_a\rho(V)A\rho(V)^*a$ back by $\rho(V)^*$. It becomes
$\sqrt2[w,Aw]$ with $w=\rho(V)^*a$. Its conditional law is that of a Haar simple bivector, independent of the fixed operator $A$. Applying the Haar inequality conditionally and then averaging gives the following bound.
 Lemma~\ref{lem:haar} therefore gives
\[
 \E\bigl(Q(\widetilde T)+Q(\widetilde U)\bigr)
 \le \frac C{n^3}\E\left(
 \norm{\widetilde T-\tau(\widetilde T)I}_{\HS}^2
 +\norm{\widetilde U-\tau(\widetilde U)I}_{\HS}^2
 +\norm{\R_0\widetilde T}_F^2+\norm{\R_0\widetilde U}_F^2\right).
\]
The Hilbert--Schmidt terms have expectation at most $C N\ell^2\eta^{-2}$. The partial traces have expectation at most $C n^{5/2}\ell^2\eta^{-12}$ by~\eqref{eq:relative-partial}, applied to the shorter independent blocks. Thus the old contribution is at most $C n^{-1/2}\ell^2\eta^{-12}$. Combine with the recent terms and negligible comparison errors. Finally apply~\eqref{eq:weight} to insert a fixed norm weight of order $j$. The weaker logarithmic power $\ell^{j+4}$ in~\eqref{eq:two-block-resolvent} leaves a reserve and proves the claim.
\end{proof}

\section{Individual-angle influence of a regularized inverse}\label{sec:influence}
The next inequality is deterministic. It avoids expanding a resolvent into covariance powers.

\begin{lemma}\label{lem:det-influence}
Let $X,Y\succeq0$ be operators on $\g$, let $a$ be a coordinate generator, and put
\[
 T=(X+Y+\eta I)^{-1},\quad M_Y=\norm Y_{\op},\quad
 D_a(\theta)=X+P_a+O_a(\theta)^*Y O_a(\theta).
\]
For $0<\eta\le1$,
\begin{equation}\label{eq:det-influence}
 \left\|\left(\partial_\theta(D_a(\theta)+\eta I)^{-1}\right)a\right\|
 \le \frac C\eta\left(1+\frac{M_Y}\eta\right)
 \left(\norm{L_aYT a}+M_Y\norm{L_aTa}\right).
\end{equation}
Consequently its square is at most
\begin{equation}\label{eq:det-influence-square}
 C\eta^{-4}(1+M_Y)^4
       \left(\norm{L_aYT a}^2+\norm{L_aTa}^2\right).
\end{equation}
\end{lemma}
\begin{proof}
Suppress the subscript $a$ on $L,O$. Let
\[
 T_0=(X+Y+P_a+\eta I)^{-1},\quad
 T_\theta=(D_a(\theta)+\eta I)^{-1},\quad
 z_0=T_0a,\quad z_\theta=T_\theta a.
\]
Sherman--Morrison gives $z_0=Ta/(1+\ip a{Ta})$. For every vector $z$,
\begin{equation}\label{eq:DeltaY}
 \norm{(O^*YO-Y)z}
 \le C\bigl(\norm{LYz}+M_Y\norm{Lz}\bigr).
\end{equation}
Indeed write the difference as $O^*Y(O-I)z+(O^*-I)Yz$ and use~\eqref{eq:coordinate-rotation} and the bound following it. The resolvent identity then yields, with $D_0=\norm{LYTa}+M_Y\norm{LTa}$,
\begin{equation}\label{eq:z-difference}
 \norm{z_\theta-z_0}
 =\norm{T_\theta(O^*YO-Y)z_0}
 \le C\eta^{-1}D_0.
\end{equation}
Put $Z=O^*YO$. Since $L$ commutes with $O$,
\[
 \norm{LZz}\le\norm{LYz}+CM_Y\norm{Lz}.
\]
Also $\partial_\theta D_a=-[L,Z]$, so the inverse derivative applied to $a$ is $T_\theta[L,Z]z_\theta$. Its norm is at most
\[
 C\eta^{-1}\left(\norm{LYz_\theta}+M_Y\norm{Lz_\theta}\right).
\]
Replace $z_\theta$ by $z_0$ using~\eqref{eq:z-difference}. The error costs at most $CM_Y\norm{z_\theta-z_0}$, since $\norm L\le1$. The $z_0$ terms are bounded by $D_0$, because $1+\ip a{Ta}\ge1$. This proves~\eqref{eq:det-influence}. Squaring and using $\eta\le1$ proves~\eqref{eq:det-influence-square}.
\end{proof}

\begin{proposition}\label{prop:influences}
Let $B$ be the left covariance of an $m$-step block, $m\le2N$, with columns $w_s$. Then
\begin{equation}\label{eq:influences}
 \frac1N\sum_{s=1}^m
 \E\norm{\left(\partial_{\theta_s}(B+\eta I)^{-1}\right)w_s}^2
 \le C n^{-1/2}\ell^8\eta^{-16}.
\end{equation}
\end{proposition}
\begin{proof}
Conjugate into the frame immediately before step $s$. It is independent of $\theta_s$ and sends $w_s$ to the fresh generator $a$. In that frame,
\[
 B\longmapsto X+P_a+O_a(\theta_s)^*YO_a(\theta_s),
\]
where $X$ is the right covariance of the past block and $Y$ is the left covariance of the future block. These two blocks, the fresh plane, and the fresh angle are mutually independent. This follows directly from the prefix formula for the left columns: past columns become right columns, while future columns have the extra factor $O_a(\theta_s)^*$.

Apply~\eqref{eq:det-influence-square} and average over the fresh plane. Proposition~\ref{prop:two-block}, with $j=4$, bounds the expectation for this single position by $C n^{-1/2}\ell^8\eta^{-16}$. The constants are uniform in the two lengths, including zero.

The change of frame is independent of $\theta_s$, so differentiation introduces no derivative of that frame. This is important: in the identity above, the only $\theta_s$-dependence is the explicit $O_a(\theta_s)^*YO_a(\theta_s)$. It follows that the derivative computed by Lemma~\ref{lem:det-influence} is exactly the conjugate of $(\partial_{\theta_s}(B+\eta I)^{-1})w_s$, not an approximation to it. The combined lengths of $X$ and $Y$ are $m-1\le2N$, so Proposition~\ref{prop:two-block} applies to every position.
 Summing over $m\le2N$ positions and dividing by $N$ absorbs at most a factor two and proves~\eqref{eq:influences}.
\end{proof}

\section{An exactly law-preserving coupling}\label{sec:coupling}
For a fixed coordinate-plane word of length $m=2N$, its endpoint map on the angle torus is $F(\theta)=R_m\cdots R_1$. The left-trivialized differential is
\[
 F^{-1}\partial_{\theta_s}F=\sqrt2w_s,\qquad
 Jz=\sqrt2\sum_s z_sw_s,\qquad JJ^*=2B.
\]
For a tangent displacement $h\in\g$, let $T_B=(B+\eta I)^{-1}$ and define
\begin{equation}\label{eq:fields}
 v_s(\theta)=-\frac1{\sqrt2}\ip{w_s}{T_Bh},\qquad
 u_s(\theta_{-s})=\frac1{2\pi}\int_0^{2\pi}
                       v_s(\theta_{-s},\alpha)\,d\alpha.
\end{equation}
Every $u_s$ has zero derivative in its own angle, so $\diver u=0$. For fixed $n,\eta$ the field is smooth on the compact angle torus. Its exact flow $\Phi_\epsilon$ is globally defined and preserves the joint uniform angle law, by the Jacobian formula
\[
 \det D\Phi_\epsilon
 =\exp\left(\int_0^\epsilon(\diver u)\circ\Phi_t\,dt\right)=1.
\]
Thus $F(\theta)$ and $F(\Phi_\epsilon(\theta))e^{\epsilon h}$ have the correct block marginals from $I$ and $e^{\epsilon h}$. Their first-order left-trivialized difference is
\begin{equation}\label{eq:endpoint-residual}
 h+Ju=\eta T_Bh+J(u-v).
\end{equation}
The flow, not an additive angle shift, is used. Nothing in this construction requires a density for the endpoint law or invertibility of $B$.

For the first-order identity, smoothness gives
\[
 F(\theta)^{-1}F(\Phi_\epsilon(\theta))
 =I+\epsilon Ju(\theta)+O(\epsilon^2).
\]
Right multiplication by $e^{\epsilon h}=I+\epsilon h+O(\epsilon^2)$ gives tangent displacement $Ju+h$. Moreover,
\[
 Jv=-\sum_s w_sw_s^*T_Bh=-BT_Bh,
 \qquad I-BT_B=\eta T_B.
\]
This verifies both the sign and normalization in~\eqref{eq:endpoint-residual}. The Riemannian norm of the tangent displacement gives the first-order distance. For fixed $n,\eta$, all functions are smooth on finitely many compact angle tori and the sphere of unit tangent directions, so the integrated squared-distance expansion has a uniform $o(\epsilon^2)$ remainder.

\subsection{Directional symmetry}
Let $Q$ be a uniform orientation-preserving signed permutation. For $n\ge5$,
\begin{equation}\label{eq:symmetry}
 \E_Q(\rho(Q)h)(\rho(Q)h)^*=\frac{\norm h^2}{N}I.
\end{equation}
To verify it, first average over determinant-one coordinate signs. Products of distinct edge coefficients have a nontrivial character supported on two or four coordinates, a nonempty proper subset when $n\ge5$. Thus off-diagonal products average to zero. Signed coordinate permutations then equalize all the diagonal entries; their sum is $\norm h^2$.

The Kac kernel is invariant under conjugation by this finite group. Construct the coupling with $h_Q=\rho(Q)h$, using $Q$ independent of the block inputs, and conjugate both endpoints back by $Q$. It is a legal coupling from the original starting points and costs no additional walk steps. In estimates below this symmetry averages the direction $h$ independently of the random covariances.

\subsection{Cost of averaging each angle component}

For a smooth scalar function on the circle, write
$f(\theta)=\sum_{k\in\mathbb Z}\widehat f_k e^{ik\theta}$. Parseval gives, for normalized Lebesgue measure,
\[
 \int|f-\bar f|^2=\sum_{k\ne0}|\widehat f_k|^2
 \le\sum_{k\ne0}k^2|\widehat f_k|^2=\int|f'|^2.
\]
Apply this to the $s$th angle component, keeping the other angles, the plane word, and $h_Q$ fixed. Integrate the resulting inequality over those inputs and sum $s$ to obtain
\begin{equation}\label{eq:circle-Poincare-full}
 \E\sum_s|u_s-v_s|^2\le\sum_s\E|\partial_{\theta_s}v_s|^2.
\end{equation}
This is applied globally on the angle torus, conditional on the other angles and the plane word; no covariance cutoff is inserted into it. Since $w_s$ has zero own-angle derivative and $\partial_{\theta_s}T_B$ is symmetric, symmetrization and Proposition~\ref{prop:influences} give
\begin{equation}\label{eq:field-cost}
 \E\sum_s|u_s-v_s|^2
 \le\frac{\norm h^2}{2N}
       \sum_s\E\norm{(\partial_{\theta_s}T_B)w_s}^2
 \le C n^{-1/2}\ell^8\eta^{-16}\norm h^2.
\end{equation}
On $\norm B\le H_*=C_*\ell$, multiply by $2H_*$ using $JJ^*=2B$.
On the complement, the global inverse bound $\norm{T_B}\le\eta^{-1}$ implies
$|u_s|+|v_s|\le C\eta^{-1}\norm h$ at every angle configuration. Hence
\[
 \norm{J(u-v)}^2\le Cm^2\eta^{-2}\norm h^2.
\]
The covariance norm tail controls this contribution by $n^{-100}\norm h^2$, with $C_*$ fixed large enough. In particular the own-angle average is not assumed to stay inside the good-norm event. Combining with~\eqref{eq:field-cost}, and deliberately weakening the powers, yields
\begin{equation}\label{eq:repair-final}
 \E\norm{J(u-v)}^2
 \le C n^{-1/2}\ell^{12}\eta^{-20}\norm h^2.
\end{equation}

The spectral residual is controlled by Proposition~\ref{prop:scalar}. Since $B$ and $D_{2N}$ are isospectral in distribution and
$\eta^2(B+\eta I)^{-2}\preceq\eta(B+\eta I)^{-1}$,
\begin{equation}\label{eq:residual-final}
 \frac{\E\norm{\eta T_Bh_Q}^2}{\norm h^2}
 =\E\tau(\eta^2T_B^2)\le\eta\E\tau(T_B)\le2\eta.
\end{equation}
Minkowski in~\eqref{eq:endpoint-residual} therefore proves the uniform infinitesimal estimate
\begin{equation}\label{eq:local-kappa}
 \lambda_n(\eta)\le\sqrt{2\eta}
       +C n^{-1/4}\ell^6\eta^{-10}.
\end{equation}
All the probabilistic estimates have been uniform for $\eta\ge n^{-1/100}$.

\subsection{Local-to-global and the choice of regularization}

We give the local-to-global step explicitly. This is the compact-group form of Oliveira's transportation argument~\cite{Oliveira}.
\begin{lemma}[Infinitesimal contraction implies global contraction]\label{lem:local-global-expanded}
Fix $n$ and a block length $m$. Suppose that, uniformly over $\|h\|=1$, there are couplings from $I$ and $e^{\epsilon h}$ with
\[
 \E d(X_\epsilon,Y_\epsilon)^2\le\epsilon^2\lambda^2+o(\epsilon^2).
\]
Then for all probability measures $\mu,\nu$ on $\SO(n)$,
\[
 W_2(\mu P_n^m,\nu P_n^m)\le\lambda W_2(\mu,\nu).
\]
\end{lemma}
\begin{proof}
The kernel is right equivariant: a walk from $x$ has endpoint $Fx$ when the corresponding walk from $I$ has endpoint $F$. Right multiplication is an isometry for $d$, so the local estimate transfers to every starting point. For any $\delta>0$, uniformity in $h$ gives a sufficiently small distance at which the local coefficient is at most $\lambda+\delta$. Subdivide a minimizing geodesic between arbitrary $x,y$ into pieces shorter than this distance. The Wasserstein triangle inequality bounds the distance between their endpoint laws by $(\lambda+\delta)$ times the total geodesic length. Let $\delta\downarrow0$.

For finitely supported initial laws, choose an optimal coupling of the finitely many starting atoms. For every pair use an optimal endpoint coupling, and mix these finitely many couplings. Its squared cost is at most $\lambda^2$ times the original squared cost. For general laws choose finite $\varepsilon$-nets and measurable nearest-net maps. Their pushforward measures approach the originals in $W_2$. Synchronous increments give
\[
 W_2(P_n^m(x,\cdot),P_n^m(y,\cdot))\le d(x,y),
\]
so the propagated finite-net approximations approach the propagated original laws as well. Passing to the limit proves the assertion without a measurable-selection problem for a continuum of endpoint couplings.
\end{proof}

For each fixed $n,\eta$, the uniform Taylor statement already verified above applies to~\eqref{eq:local-kappa}. Consequently
\begin{equation}\label{eq:global-kappa-expanded}
 \kappa_n\le\lambda_n(\eta)
 \le\sqrt{2\eta}+Cn^{-1/4}\ell^6\eta^{-10}.
\end{equation}
The small scale used in geodesic subdivision may depend on $n$ and $\eta$. The proof takes that infinitesimal limit for each fixed $n$ before making the asymptotic parameter choice; no uniform-in-$n$ Taylor neighborhood is assumed.

Choose
\begin{equation}\label{eq:eta-choice}
 \eta=n^{-1/100}.
\end{equation}
Then~\eqref{eq:local-kappa} becomes
\[
 \kappa_n\le\sqrt2\,n^{-1/200}+C\ell^6 n^{-3/20}
 \le2n^{-1/200}\le n^{-1/300}
\]
for sufficiently large $n$. This proves~\eqref{eq:main-kappa}.
The group has diameter at most $\pi\sqrt n$: a rotation has a skew logarithm whose planar angles lie in $[-\pi,\pi]$, with squared Hilbert--Schmidt norm at most $n\pi^2$. Haar stationarity and block contraction imply
\begin{equation}\label{eq:Witerate}
 \sup_xW_2(P_n^{2Nj}(x,\cdot),\pi_n)
 \le\pi\sqrt n\,n^{-j/300}.
\end{equation}
For any fixed $A>0$, taking $j=\lceil300(A+1)\rceil$ gives accuracy at most $n^{-A}$ for all sufficiently large $n$. In particular fixed-accuracy Wasserstein mixing takes $O(N)$ steps.

\section{Total variation in a constant number of cores}\label{sec:tv}
We include the transfer to specify every dependency. Use the same regularization~\eqref{eq:eta-choice}.

\subsection{Few deficient directions and exact scalar means}
For a fresh core covariance $C=D_{2N}$,
\begin{equation}\label{eq:badfraction}
 \delta_n:=\E\tau\one_{\{C<\eta I\}}
 \le 2\eta\E\tau(C+\eta I)^{-1}\le4\eta.
\end{equation}
In~\eqref{eq:badfraction} the notation $C<\eta I$ denotes its spectral projection onto eigenvalues below $\eta$. The first inequality is the pointwise scalar inequality
$\one_{\{x<\eta\}}\le2\eta/(x+\eta)$ for $x\ge0$.

For every bounded real spectral function $f$, covariance equivariance under signed permutations gives
\begin{equation}\label{eq:cov-symmetry}
 \E f(C)=\E\tau(f(C))I.
\end{equation}
Indeed simultaneous conjugation of all coordinate rotations by a signed permutation preserves their input law. Each derivative column is conjugated, up to a sign that vanishes in its rank-one projection. Hence the law of $C$ is invariant. Its expected symmetric spectral function commutes with the signed-permutation group, whose average on symmetric operators is $A\mapsto\tau(A)I$ by~\eqref{eq:symmetry} and spectral decomposition. This proves~\eqref{eq:cov-symmetry} without full Haar invariance.

Concatenate a fixed number $b$ of independent $2N$-step cores. Let $C_j,H_j$ be the right covariance and endpoint of core $j$, and put
\[
 P_j=\one_{\{C_j<\eta I\}},\quad
 F_{>j}=H_b\cdots H_{j+1},\quad
 Q_j=\rho(F_{>j})P_j\rho(F_{>j})^*.
\]
The complete path of core $j$ is independent of all later cores; no independence between $C_j$ and $H_j$ is asserted or needed. Conditional on the full paths of later cores,
\begin{equation}\label{eq:conditional-projection}
 \E[Q_j\mid\text{later cores}]=\delta_n I\preceq\beta_n I,
 \qquad \beta_n=4n^{-1/100}.
\end{equation}
The full covariance is the sum of the transported core covariances and therefore satisfies
\[
 D_{2Nb}\succeq\eta\left(bI-\sum_{j=1}^bQ_j\right).
\]

For completeness, conditional Golden--Thompson gives
\[
 \E\Tr\exp\left(\theta\sum_jQ_j\right)
 \le N[1+\beta_n(e^\theta-1)]^b.
\]
Add the projections in reverse chronological order, so the earlier-added sum is measurable with respect to the later core paths. The identity $e^{\theta Q}=I+(e^\theta-1)Q$ for projections and~\eqref{eq:conditional-projection} justify each conditional step. Taking $\theta=\log(1/(2\beta_n))>0$ yields
\begin{equation}\label{eq:proj-tail}
 \PP\left\{\lambda_{\max}\left(\sum_jQ_j\right)>b/2\right\}
 \le N(2e\beta_n)^{b/2}.
\end{equation}
Choose, for example, the fixed integer $b=10000$. Then the right side is
$N(8e)^{5000}n^{-50}\le N^{-10}$ for sufficiently large $n$. Thus
\begin{equation}\label{eq:whole-cov}
 \PP\{\lambda_{\min}(D_{2Nb})<b\eta/2\}\le N^{-10}.
\end{equation}
The same holds for the left covariance by pointwise conjugacy. No actual separator steps have been inserted between the cores.

\subsection{A cutoff integration-by-parts modulus}
\begin{lemma}\label{lem:tv-modulus}
If the left covariance $B_m$ of an $m$-step block satisfies
$\PP\{\lambda_{\min}(B_m)<2a\}\le\delta$, then
\begin{equation}\label{eq:tv-modulus}
 \norm{P_n^m(x,\cdot)-P_n^m(y,\cdot)}_{\TV}
 \le2\delta+C(m/a)^2d(x,y).
\end{equation}
No density assumption on the full endpoint law is required.
\end{lemma}
\begin{proof}
We give the construction and the translation estimate separately.

\paragraph{Derivative of the covariance.}
Fix a plane word and its endpoint $F(\theta)$. For the left columns, differentiating the inverse prefix gives
\begin{equation}\label{eq:TV-columns-expanded}
 \partial_{\theta_s}w_t=-\sqrt2[w_s,w_t]\quad(t>s),
 \qquad \partial_{\theta_s}w_t=0\quad(t\le s).
\end{equation}
The generator $\ad_{\sqrt2w_s}$ is an orthogonal conjugate of a coordinate generator and has operator norm one. Therefore $\|\partial_sw_t\|\le1$ and
\[
 \partial_s B=\sum_{t>s}\bigl((\partial_sw_t)w_t^*+
                                    w_t(\partial_sw_t)^*\bigr).
\]
The nuclear norm of $uv^*$ is $\|u\|\|v\|$. There are at most $m$ summands and $\|w_t\|=1$, whence
\begin{equation}\label{eq:TV-nuclear-expanded}
 \|\partial_sB\|_{S_1}\le2m,
 \qquad \|\partial_sB\|_{\op}\le2m.
\end{equation}

\paragraph{A globally smooth cutoff.}
Choose fixed smooth functions $\varphi,\zeta:\mathbb R\to[0,1]$ with bounded derivatives, such that
\[
 \varphi=1\text{ on }(-\infty,1],\quad\varphi=0\text{ on }[2,\infty),
 \quad
 \zeta=1\text{ on }(-\infty,1/4],\quad\zeta=0\text{ on }[1/2,\infty).
\]
For example, let $h(t)=0$ for $t\le0$ and $h(t)=e^{-1/t}$ for $t>0$, and put
$\omega(t)=h(1-t)/(h(1-t)+h(t))$. Then $\varphi(x)=\omega(x-1)$ and
$\zeta(x)=\omega(4x-1)$ have the required properties.
Set
\begin{equation}\label{eq:TV-cutoff-expanded}
 \chi(\theta)=\zeta\bigl(\Tr\varphi(B(\theta)/a)\bigr).
\end{equation}
If $B\succeq2aI$ then $\chi=1$. If $\lambda_{\min}(B)\le a$, one term in the trace is one, so $\chi=0$. Continuity of $\varphi$ at $1$ shows that the cutoff is zero on a fixed neighborhood of the latter set as well: $\varphi$ remains greater than $1/2$ immediately to the right of $1$. This ensures smooth extension of the inverse field below.

For a smooth scalar function of a symmetric matrix,
\[
 \partial_s\Tr\varphi(B/a)
 =a^{-1}\Tr\bigl(\varphi'(B/a)\partial_sB\bigr).
\]
One sees this first at simple spectrum by differentiating the eigenvalue sum; continuity gives the same formula at repeated eigenvalues. Equivalently, the derivative of the spectral matrix function has divided differences off the diagonal, whose trace vanishes. The function itself is smooth on the symmetric matrices because $\varphi$ is smooth. Pairing operator norm with nuclear norm in~\eqref{eq:TV-nuclear-expanded} gives
\begin{equation}\label{eq:TV-cutoff-derivative-expanded}
 |\partial_s\chi|\le Cm/a.
\end{equation}
No factor $N$ is introduced by the trace.

\paragraph{A vector field representing a group translation.}
For fixed $h\in\g$, define on $B\succ aI$
\[
 V_s=\frac1{\sqrt2}\langle w_s,B^{-1}h\rangle.
\]
Since $Jz=\sqrt2\sum_s z_sw_s$, we have $JV=h$. The field $\chi V$ extends smoothly by zero over the full torus, because the cutoff vanishes in a neighborhood of the singular set. On its support,
\[
 |V_s|\le C\|h\|/a,
 \qquad
 |\partial_sV_s|\le Cm\|h\|/a^2.
\]
The second bound follows from $\partial_sw_s=0$ and
$\partial_sB^{-1}=-B^{-1}(\partial_sB)B^{-1}$. Hence
\begin{equation}\label{eq:TV-divergence-expanded}
 |\diver(\chi V)|
 \le\sum_s\bigl(|\partial_s\chi|\,|V_s|+\chi|\partial_sV_s|\bigr)
 \le C(m/a)^2\|h\|.
\end{equation}

\paragraph{Integration by parts for a submeasure.}
For a smooth function $f$ on the group write
$X_hf(g)=\left.\frac{d}{dt}f(ge^{th})\right|_{t=0}$. The equality $JV=h$ gives
\[
 \sum_s V_s\partial_s(f\circ F)=X_hf(F)
\]
where the cutoff is nonzero. Periodic integration by parts, valid globally for the smooth field $\chi V$, therefore yields
\begin{equation}\label{eq:TV-IBP-expanded}
 \E_\theta[\chi X_hf(F)]
 =-\E_\theta[f(F)\diver(\chi V)].
\end{equation}
Average over the plane word too. Let $\mu=P_n^m(I,\cdot)$ and let $\mu_\chi$ be its endpoint pushforward weighted by $\chi$. Since $0\le\chi\le1$ and $\chi=1$ on $B\succeq2aI$, the positive measure $\mu-\mu_\chi$ has mass at most $\delta$. Equation~\eqref{eq:TV-IBP-expanded} bounds every directional distributional derivative of $\mu_\chi$ by $C(m/a)^2\|h\|$ in variation norm. It makes no assertion that the full law $\mu$ has a density.

To obtain a finite translation, put $f_t(g)=f(ge^{th})$. The two exponentials in the same direction commute, so
$\partial_tf_t(g)=X_hf_t(g)$. Integrating~\eqref{eq:TV-IBP-expanded} from $0$ to $1$, using~\eqref{eq:TV-divergence-expanded}, gives for $\|f\|_\infty\le1$
\[
 \left|\int f(ge^h)\,d\mu_\chi(g)-\int f(g)\,d\mu_\chi(g)\right|
 \le C(m/a)^2\|h\|.
\]
Smooth test functions determine the variation norm of finite Borel measures on a compact smooth manifold: regularity reduces to continuous functions, which admit uniform smooth approximation in finitely many charts and a partition of unity. Restoring the omitted measure and its translate costs at most $2\delta$ in the deliberately loose probability-TV convention used here.

Finally $P_n^m(x,\cdot)$ is the right translate of $\mu$ by $x$. Translate both endpoint laws by $x^{-1}$ and choose a minimizing skew logarithm $h$ of $yx^{-1}$. A planar normal form for a rotation shows that such a logarithm exists and $\|h\|=d(x,y)$ for the bi-invariant metric. The preceding bound proves~\eqref{eq:tv-modulus}.
\end{proof}

Apply the lemma with $m=2Nb$, $a=b\eta/4$, and $\delta=N^{-10}$ from~\eqref{eq:whole-cov}. Since $m/a=8N/\eta$,
\[
 \norm{P_n^{2Nb}(x,\cdot)-P_n^{2Nb}(y,\cdot)}_{\TV}
 \le2N^{-10}+CN^2\eta^{-2}d(x,y).
\]
Integrate against any coupling of initial measures, minimize its first-moment cost, and use $W_1\le W_2$ and Haar stationarity. For every probability $\mu$,
\begin{equation}\label{eq:TVtransfer}
 \norm{\mu P_n^{2Nb}-\pi_n}_{\TV}
 \le2N^{-10}+CN^2\eta^{-2}W_2(\mu,\pi_n).
\end{equation}
By~\eqref{eq:Witerate}, a fixed number $j=3300$ of initial cores gives Wasserstein error at most $n^{-10}$ for sufficiently large $n$. With $\eta=n^{-1/100}$,
\[
 \sup_x\norm{P_n^{2N(j+b)}(x,\cdot)-\pi_n}_{\TV}
 \le2N^{-10}+CN^2 n^{2/100}n^{-10}\le Cn^{-5}.
\]
The total number of steps is $2N(j+b)=O(N)$, completing Theorem~\ref{thm:main}.

\begin{proposition}[Dimensional lower bound]
For $t<N$, the endpoint law from any point is singular with respect to Haar and has total-variation distance one from it.
\end{proposition}
\begin{proof}
For a fixed plane word, the endpoint is the smooth image of a $t$-dimensional compact torus in the $N$-dimensional group. Its image has Hausdorff dimension at most $t$ and zero Riemannian volume if $t<N$. There are finitely many plane words. Haar measure is normalized Riemannian volume; a finite union of these images is Haar-null. Translation preserves null sets.
\end{proof}

\section{What the theorem does and does not say about cutoff}\label{sec:cutoff}
The distinction between a mixing-time scale and the shape of the transition is important. A bound of order $n^2$ is not, by itself, a theorem about absence of cutoff.

\begin{definition}[Total-variation cutoff and pre-cutoff]
Write $d_n(t)=\sup_x\|P_n^t(x,\cdot)-\pi_n\|_{\TV}$. A sequence has total-variation cutoff if
\begin{equation}\label{eq:cutoff-definition}
 \lim_{n\to\infty}\frac{t_{\TV}^{(n)}(\epsilon)}
 {t_{\TV}^{(n)}(1-\epsilon)}=1
 \quad\hbox{for every }0<\epsilon<1/2.
\end{equation}
It has pre-cutoff if
\begin{equation}\label{eq:precutoff-definition}
 \sup_{0<\epsilon<1/2}\limsup_{n\to\infty}
 \frac{t_{\TV}^{(n)}(\epsilon)}{t_{\TV}^{(n)}(1-\epsilon)}<\infty.
\end{equation}
These definitions are in terms of the sequence of kernels, not of one fixed dimension; see, for example,~\cite{HermonPeres}.
\end{definition}

\begin{proof}[Proof of Corollary~\ref{cor:precutoff}]
The dimensional lower bound in Section~\ref{sec:tv} gives $d_n(t)=1$ for $t<N$. The construction there uses $j=3300$ Wasserstein cores and $b=10000$ smoothing cores. At the deterministic time
\[
 2N(j+b)=26600N
\]
its TV distance is at most $Cn^{-5}$, which tends to zero. For each fixed $\epsilon\in(0,1)$ this yields
\[
 N\le t_{\TV}^{(n)}(\epsilon)\le26600N
\]
for all sufficiently large $n$. The upper constant is the same for every fixed accuracy; only the sufficiently-large-$n$ threshold can depend on $\epsilon$. The ratio in~\eqref{eq:precutoff-definition} consequently has limit superior at most $26600$, uniformly in $\epsilon$. This proves pre-cutoff.

The argument supplies neither a limit for that ratio nor two accuracies whose ratio is bounded away from one. Thus it supplies neither a proof of cutoff nor a proof of its absence.
\end{proof}

Equivalently, the estimates bound the mixing region between two constant multiples of $N$, but they do not determine the width of that region on the $N$ scale. A cutoff at some $t_n\asymp N$ with window $o(N)$ is compatible with all the estimates. So is a non-sharp transition inside a constant-multiple range. A cutoff at a scale $t_n$ with $t_n/N\to\infty$, such as $n^2\log n$, would be incompatible with the quadratic upper bound, because the distance is already tending to zero by $26600N=o(t_n)$. Ruling out that larger scale is not the same statement as ruling out cutoff.

There is also no obstruction from confusing the relaxation scale with $N$. In the present discrete-time normalization,
\[
 t_{\mathrm{rel}}^{(n)}=\gamma_n^{-1}
 =\frac{2n(n-1)}{n+2}\asymp n,
 \qquad
 \frac{t_{\mathrm{rel}}^{(n)}}{N}=\frac4{n+2}\longrightarrow0.
\]
Thus the relaxation time is smaller than the quadratic mixing scale. We do not use any converse cutoff criterion based on this comparison.

\subsection{Two elementary examples explaining the distinction}
The sphere cutoff theorem from a coordinate start in the preprint of Jain--Mizgerd~\cite{JM26}, and the uniform-plane matrix cutoff theorem of Hough--Jiang~\cite{HJ17}, concern different observations or different kernels. Neither settles the full coordinate-plane matrix cutoff question considered here.

The following examples are not models for the detailed behavior of Kac's walk. They demonstrate directly that the same quadratic order of a mixing time can coexist with either behavior in~\eqref{eq:cutoff-definition}.

\begin{example}[Quadratic mixing without cutoff]\label{ex:no-cutoff}
On an $n$-point space with uniform law $\pi$, let
\[
 K_n=(1-n^{-2})I+n^{-2}\Pi,
\]
where $\Pi$ replaces the current state by an independent uniform state. Since $\Pi^2=\Pi$,
\[
 K_n^t=(1-n^{-2})^tI+\bigl(1-(1-n^{-2})^t\bigr)\Pi.
\]
Its worst-start distance is exactly
\[
 d_n(t)=(1-1/n)(1-n^{-2})^t.
\]
Therefore, for each fixed $0<\epsilon<1$,
\[
 t_n(\epsilon)\sim n^2\log(1/\epsilon).
\]
For $0<\epsilon<1/2$, the limiting ratio is
$\log(1/\epsilon)/\log(1/(1-\epsilon))>1$, not one. The mixing time is quadratic at every fixed accuracy, but there is no cutoff.
\end{example}

\begin{example}[Quadratic mixing with cutoff]\label{ex:with-cutoff}
On $\{-1,1\}^n$, independently for every coordinate and every time step, replace that coordinate by a fresh fair sign with probability
\[
 p_n=\frac{\log n}{2n^2}.
\]
Otherwise leave it unchanged. The stationary law is the product of fair signs. Symmetry makes every starting state equivalent; take the all-positive state. At time $t$, the signs are independent with common mean
\[
 m_t=(1-p_n)^t.
\]
The density relative to stationarity is $\prod_i(1+m_tx_i)$, so its chi-squared divergence is
\[
 \chi^2_t=(1+m_t^2)^n-1.
\]
At $t=\lceil(1+\delta)n^2\rceil$, for fixed $\delta>0$, we have
$nm_t^2=n^{-\delta+o(1)}\to0$. Cauchy--Schwarz gives
$d_n(t)\le\tfrac12\sqrt{\chi^2_t}\to0$.

For $t=\lfloor(1-\delta)n^2\rfloor$ with $0<\delta<1$, let $S_t$ be the sum of the signs. Under the time-$t$ law its mean is $nm_t$ and variance at most $n$; at stationarity the mean is zero and variance $n$. Here $nm_t^2=n^{\delta+o(1)}\to\infty$. The event $S_t\ge nm_t/2$ has probability tending to one under the time-$t$ law and to zero at stationarity, by Chebyshev's inequality. Thus $d_n(t)\to1$.

For every fixed $\delta>0$, the distance is close to one before $(1-\delta)n^2$ and close to zero after $(1+\delta)n^2$. Monotonicity then implies~\eqref{eq:cutoff-definition}, with cutoff time $n^2$. Again the mixing time is quadratic, but now cutoff occurs.
\end{example}

\section{Quantitative scope and reproducibility}\label{sec:scope-final}
The proof is asymptotic. Its exponents and fixed numbers of cores were chosen to leave ample room in each estimate, not to produce a practical mixing prescription at a specified moderate dimension. The sufficiently-large-$n$ threshold and the absolute constants in the resolvent estimates are not numerically evaluated.

All transport cores use exactly $2N$ genuine coordinate-plane updates. The suffix length $R=\lceil n^{5/4}\rceil$ appears only in proofs of expectation estimates. Likewise, the inactive rotations in the relative-Haar colored process are an auxiliary comparison device, not uncounted updates in the coupling. The TV stage consists of $b$ additional cores, each of exactly $2N$ steps, with no intervening free scrambling.

The proof does not assume that the entire finite-time endpoint law has a density. At every finite time there is, for example, a positive-probability event that all selected planes coincide, leaving mass on a one-dimensional subgroup. The regularized-inverse coupling is defined even on that event. The TV smoothing argument separately discards an explicitly bounded exceptional submeasure before using an unregularized inverse and restores its mass at the end.

The Hilbert--Schmidt chord distance is at most $d$: integrate the ambient velocity along any group path and minimize its Riemannian length. Wasserstein distance decreases when its ground metric decreases, and $W_1\le W_2$. These facts prove the metric variants in Theorem~\ref{thm:main}. The proof does not identify a limiting TV profile, a sharp constant multiplying $n^2$, or a cutoff window.

\subsection*{Classical inputs and what is proved here}
The exact group spectral-gap theorem~\eqref{eq:classical-gap} is the cited theorem of Carlen--Carvalho--Loss; its full proof is not reproduced. The Golden--Thompson inequality is also used as a classical finite-dimensional result. Their normalization and all their uses are stated explicitly. The representation averaging consequence, Gaussian contraction estimates, characteristic stability inequality, matrix-noise accumulation, rational-function approximation, conditional covariance amplification, and the transport-to-TV argument are proved in this paper. No earlier manuscript from this project, computational output, or random-matrix limit theorem is an additional hypothesis.

\subsection*{AI-assistance disclosure}
Substantial assistance from the generative AI system designated GPT-6 Astra was used in exploring the proof strategy, drafting mathematical arguments and exposition, and preparing algebra-check programs. The mathematical justification is the argument in this manuscript, not the output of those programs. Responsibility for the submitted content rests with the human author. The separate review materials record the internal checks and their limitations.

\end{document}